\documentclass[11pt]{amsart} \usepackage[left=1.3in, right=1.2in, top=1in, bottom=1in, includefoot]{geometry}

\usepackage{soul,xcolor}
\usepackage{bbm}

\usepackage{dsfont}

\usepackage{amsmath}
\usepackage{amsfonts}
\usepackage{amssymb}

\usepackage{amssymb}
\usepackage{enumitem}

\usepackage[nocompress]{cite}

\usepackage{hyperref}

\hypersetup{
	colorlinks   = true, 
	urlcolor     = blue, 
	linkcolor    = blue, 
	citecolor   = red 
}

\usepackage{caption}

\newtheorem{thm}{Theorem}[section]
\newtheorem{prop}[thm]{Proposition}
\newtheorem{lem}[thm]{Lemma}
\newtheorem{cor}[thm]{Corollary}
\newtheorem{assump}{Assumption}

\theoremstyle{definition}
\newtheorem{definition}[thm]{Definition}

\newtheorem{prob}{Problem}

\DeclareRobustCommand{\gobblefive}[5]{}
\newcommand*{\SkipTocEntry}{\addtocontents{toc}{\gobblefive}}

\theoremstyle{remark}

\newtheorem{remark}[thm]{Remark}

\numberwithin{equation}{section}

\newcommand{\R}{\mathbb{R}}  
\newcommand{\N}{\mathbb{N}}  

\usepackage{float}
\usepackage{amsmath}
\usepackage{amssymb} 
\usepackage{mathtools}

\newcommand\norm[1]{\lVert#1\rVert}
\renewcommand\d{\:\mathrm{d}}

\newcommand{\scalarprod}[2]{\big({#1},{#2}\big)}

\newcounter{rtaskno}

\newcounter{rsubtaskno}

\usepackage{scrextend}

\usepackage{xcolor}

\begin{document}
	
	
	\title[Existence and uniqueness]{Weak solvability of elliptic variational inequalities coupled with a nonlinear differential equation} 


\author[N. Skoglund Taki]{Nadia Skoglund Taki}
\address{Center for Modeling of Coupled Subsurface Dynamics \\ Department of Mathematics\\ University of Bergen\\ Postbox 7800\\ 5020 Bergen\\ Norway}
\date{\today}
\email{Nadia.Taki@uib.no}






\begin{abstract}
In this paper we establish existence, uniqueness, and boundedness results for an elliptic variational inequality coupled with a nonlinear ordinary differential equation. 
Under the general framework, we present a new application modelling the antiplane shear deformation of a static frictional adhesive contact problem. The adhesion process has been extensively studied, but it is usual to assume a priori that the intensity of adhesion is bounded by introducing truncation operators. The aim of this article is to remove this restriction.\\
\indent The proof is based on an iterative approximation scheme showing that the problem has a unique solution. A key ingredient is finding uniform a priori bounds for each iterate.
These are obtained by adapting versions of the Moser iteration to our system of equations.
\end{abstract}


\maketitle

\section{Introduction}\label{sec:intro}
\noindent
Let $\Omega \subset \R^d$ for $d\geq 2$ be an open bounded domain with Lipschitz boundary $\partial \Omega$, where $\partial \Omega$ is split into three disjoint parts $\partial \Omega= \Gamma_D\cup\Gamma_N \cup\Gamma_C$.  This work concerns the study of a fully coupled elliptic variational inequality with a nonlinear differential equation of the form
\begin{subequations}\label{eq:eq1}
\begin{align}\label{eq:eq1a}
\int_\Omega \mu \nabla u(t) \cdot \nabla (v-u(t)) \d x &+ \int_{\Gamma_C} [\varphi(\beta(t), u(t),  v)-\varphi(\beta(t), u(t), u(t))] \d \sigma\\ \notag
&+ \int_{\Gamma_C} [\phi(\beta(t), u(t),  v)-\phi(\beta(t), u(t), u(t))] \d \sigma\\ \notag
&\geq \int_{\Omega} f_0(t) (v-u(t)) \d x+ \int_{\Gamma_N} f_N(t)   (v-u(t)) \d \sigma, \\
\beta(t) &= \beta_0 + \int_0^t  H(\beta(s), u(s)) \d s \label{eq:eq1b}
\end{align}
\end{subequations}
for all $v\in H^1_0(\Omega)$ and a.e. $t\in (0,T)$ with a time $T>0$.  Here, the functions $\varphi$, $\phi$, and $H$ 
have nonlinear structures that will be specified later with physical applications in mind, see  \hyperref[assumptionvarphi]{$(H1)$}-\hyperref[assumptionH]{$(H3)$} and Appendix \ref{appendix:application}. 
\\
\indent
The structures on $\varphi$, $\phi$, and $H$ are new and therefore give rise to new applications of \eqref{eq:eq1}. In Appendix \ref{appendix:application}, we introduce a new frictional adhesion contact problem between a linearly elastic cylinder and an obstacle. In the case of a long cylinder, the problem can be reduced to a two-dimensional problem, that is, $d=2$ (see Appendix \ref{appendix:application}). This is an example of an elliptic variational inequality of a simplified version of a friction problem in elasticity (see \cite[p.7-9]{Glowinski1981}). The adhesion process on the contact surface is modelled through a nonlinear ODE and describes the active bonds. In this setting, \eqref{eq:eq1a}-\eqref{eq:eq1b} models the deformation $u$ and intensity of adhesion $\beta$. This process is important in industry but may also describe layered rocks \cite{Sofonea2002,Dumont2000,Goeleven1999}. For a detailed description of the new application, we refer the reader to Appendix \ref{appendix:application}. Moreover, a second application to antiplane shear deformation is the interaction of faults that includes frictional forces \cite{Ionescu2005}. An example of geophysical faults that are categorized in the antiplane shear setting is the so-called \textit{normal} faults. These types of faults are important in studying the early phase of earthquakes (see \cite{Ionescu2005} for more details). Other applications are, for instance, studying cracks (referred to as \textit{mode III} cracks) \cite{Knowles1981} and phase transition in the microstructure of the material (for example, between ionic solids) \cite{Swart1992} (see also the review \cite[Section 6]{Horgan1995} for more applications). 
Focusing on the elliptic part of the equation allows us to incorporate fully nonlinear evolution equations for the adhesion process. To include these nonlinearities, our proof differs from those found in the literature (see Section \ref{sec:former}). It was therefore an important step to first understand these changes. Ultimately, the idea is to extend the framework to fully describe fractures and faults, where we add the effects of friction, adhesion, and wear. It also requires extending the presented work to include the strain tensor used in the usual elasticity problems and dynamics. 
\\
\indent
The main purpose of this paper is to prove that the pair $(u,\beta)$ is a unique solution to \eqref{eq:eq1}. The precise statement is given in Theorem \ref{thm:mainresult}. A key difficulty of the proof is needing global a priori bounds in the iterative decoupling scheme of \eqref{eq:eq1}. These are found by adapting versions of the Moser iteration from the books \cite{Drabek1997,Gilbarg2001}.

\subsection{Former existence, uniqueness and boundedness results} \label{sec:former}
There are former existence and uniqueness results for special cases of \eqref{eq:eq1}. In particular, when $\varphi$ is independent of $\beta$ and $\phi \equiv 0$, the system is reduced to the problem studied in \cite[Section 9.2]{Sofonea2009} and \cite{Sofonea2020,Migorski2010}. In comparison to these results, we also establish boundedness. A fundamental problem to prove existence and uniqueness of \eqref{eq:eq1} is that we need $(u,\beta)$ to be bounded.
In, e.g., \cite{Sofonea2002}, they simplify the problem by using cut-off functions (see Remark \ref{remark:comparion}). This allows them to use classical theory for variational inequalities. We consider a more general version of this problem where we overcome the issue of boundedness via a Moser iteration scheme. The cut-off functions are also used in, e.g., \cite{Lerguet2008} and \cite{Sofonea2006}, but their applications are, respectively, to the usual frictional elastic and viscoelastic problems (see Remark \ref{remark:comparion} on this point). Furthermore, \cite[Theorem 3.1]{Winkert2014} provides boundedness results for a variational inequality without the coupling of an ODE, which is based on the framework in \cite{Drabek1997}. We will also utilize the techniques from \cite{Drabek1997}, but we require a finely-tuned Moser scheme to handle the coupled system. 

\subsection{Contributions and outline} The novelties of this paper are:
\begin{itemize}
\item Existence and uniqueness of \eqref{eq:eq1}, where the proof is based on an iterative decoupling approach that directly gives rise to a numerical method.
\item Boundedness of the solution $(u,\beta)$ to \eqref{eq:eq1}.
\item A new application of \eqref{eq:eq1} for $d=2$ modelling elastic antiplane shear deformation with frictional forces and adhesion. 
\end{itemize}
\indent
The paper is organized as follows. In Section \ref{sec:prelim}, we introduce the function spaces, preliminary results, trace theorems, and a Sobolev embedding theorem. The main result is presented in Section \ref{sec:mainresult}. Section \ref{sec:proof_of_mainresult} is devoted to the proof of the main result. The strategy consists of constructing an approximated solution to \eqref{eq:eq1} using an iterative scheme decoupling \eqref{eq:eq1a} and \eqref{eq:eq1b}. For this, we need the preliminary results stated in Section \ref{sec:prelim} and gaining control of the solution in $L^\infty$ in space. These estimates are obtained by constructing a finely-tuned Moser iteration scheme for each iterate (see Remark \ref{remark:moser} for more details). It is then left to show that these estimates are bounded by a constant independent of the iteration step. From these results, we can obtain strong convergence of the solution and pass to the limit to verify that the approximated problem converges to \eqref{eq:eq1}. In the last step, we show that the solution is unique. Finally, Appendix \ref{appendix:application} includes a new application fitting our framework.

\subsection{Notation}
\begin{itemize}
\item Let $0<T<\infty$ be the maximal time.
\item Let $d$ denote the dimension. A point in $\R^d$ is denoted by $x = (x_i)$,  $1 \leq i \leq d$.
\item We denote $|\cdot|$ as the Lebesgue measure on $\R^d$. We also denote $|\cdot|$ and $\cdot$  as the norm and inner product of $\R^d$, respectively.
\item $\Omega \subset \R^d$ is a bounded open connected subset with a Lipschitz boundary $\partial \Omega$. We split $\partial \Omega$ into three disjoint parts; $\Gamma_D$, $\Gamma_N$, and $\Gamma_C$ with $|\Gamma_D|>0$, $|\Gamma_C|>0$, i.e., nonzero Lebesgue measure, but $\Gamma_N$ is allowed to be empty.
\item $\nu$ denotes the outward normal on $\partial \Omega$.
\item We denote $\Bar{\Omega} = \Omega \cup \partial \Omega$.
\item We will denote $c$ as a positive constant, which might change from line to line. Let $c_0=c(d,|\Bar{\Omega}|)$ be a  positive constant.
\item The positive and negative part of a function $v$ is defined as $v_+ = \max\{v,0\}$ and $v_- = \max\{-v,0\}$, respectively, such that $v = v_+ - v_-$.
\item  For a set $A \subset \R^d$, let $\mathds{1}_A(x)$ denote the indicator function, which is equal to $1$ if $x\in A$ and $0$ if $x \not\in A$.
\end{itemize}

\tableofcontents

\section{Preliminaries}\label{sec:prelim}
In this section, we present the function spaces and some fundamental results. For further information, we refer to standard textbooks, e.g., \cite{Evans2010,Adams2003}.
\subsection{Function spaces}\label{sec:funcspace}
Let  $L^p(\Omega)$ denotes the space of Lebesgue $p$-integrable functions equipped with the norm $ \norm{v}_{L^p(\Omega)}  = \big( \int_\Omega |v|^p \d x \big)^{1/p}$ for any $1\leq p <\infty$. With the usual modifications for $L^\infty(\Omega)$. We define the space
\begin{equation*}
H^1(\Omega) = \{v \in L^2(\Omega) : \text{ the weak derivatives } \frac{\partial v }{\partial x_i}\text{ exist in } L^2(\Omega) , \ 1\leq i \leq d\}
\end{equation*}
with its usual norm $\norm{v}_{H^1(\Omega)} = \norm{v}_{L^2(\Omega)} + \norm{\nabla v}_{L^2(\Omega)}$. For the displacement, we use the Hilbert space 
\begin{align*}
    H^1_0(\Omega) = \{ v \in H^1(\Omega) : v=0 \text{ on } \Gamma_D\}
\end{align*}
equipped with the norm $\norm{\cdot}_{H^1(\Omega)}$, where we denote the trace of $v$ on $\Gamma_D$ as $v$. We denote the inner product on $H^1_0(\Omega)$ as $(\cdot,\cdot)_{H^1(\Omega)} = \norm{\cdot}_{H^1(\Omega)}^2$. Since $|\Gamma_D|>0$,  the Poincar\'{e} inequality 
\begin{align*}
\norm{v}_{L^2(\Omega)} \leq c_0 \norm{\nabla v}_{L^2(\Omega)}  \ \text{ for all } \ v\in H^1_0(\Omega).
\end{align*}
implies
\begin{align*}
\norm{v}_{H^1(\Omega)} =  \norm{v}_{L^2(\Omega)} + \norm{\nabla v}_{L^2(\Omega)} \leq (1+c_0) \norm{\nabla v}_{L^2(\Omega)}
\end{align*}
for all $v\in H^1_0(\Omega)$. 
\\
\indent We next introduce the time-dependent spaces. The following short-hand notation will come in handy.
\begin{definition}\label{def:Bochner}
Let $X$ be a Banach space and $T > 0$. The space $L^\infty(0,T;X)$ consists of all measurable functions $v : [0, T] \rightarrow X$ such that
\begin{align*}
\underset{t\in(0,T)}{\mathrm{ess \: sup}} \norm{v(t)}_X < \infty.
\end{align*}
For brevity, we use the standard short-hand notation
\begin{equation*}
L^\infty_t X = L^\infty(0, t; X)
\end{equation*}
for all $t \in [0,T]$.
\end{definition}
We denote the space of continuous functions defined on $[0,T]$ with values in $X$ by 
\begin{equation*}
C([0,T]; X) = \{h : [0,T] \rightarrow X : h \text{ is continuous and } \norm{h}_{L^\infty_TX} < \infty\}.
\end{equation*}

\subsection{Trace theorems and a Sobolev embedding theorem} 
Let $\sigma$ denote the $(d-1)$-dimensional surface measure on $\Gamma\subset \partial \Omega$, then we may define the Lebesgue spaces on $\Gamma$ $L^p(\Gamma)$ with $1\leq p < \infty$ equipped with the norm
\begin{align*}
\norm{v}_{L^p(\Gamma)} = \Big( \int_\Gamma |v|^p \d \sigma \Big)^{1/p}
\end{align*}
and the usual modification for $p=\infty$. 
The trace on  $\Gamma \subset \partial \Omega$ is usually denoted by $\gamma (v)$, but we will drop this notation. Moreover, the following proposition will be useful to prove the main result.
\begin{prop}\label{prop:trace}
Let $\Omega \subset \R^d$ be an open bounded domain with $\Gamma\subset \partial \Omega$ being Lipschitz, then for a constant $c_0>0$ we have
\begin{align*}
\norm{v}_{L^2(\Gamma)}   \leq c_0\epsilon \norm{\nabla v}_{L^2(\Omega)} + \frac{c_0}{\epsilon}\norm{v}_{L^2(\Omega)}
\end{align*}
for all $v\in H^1(\Omega)$ and $\epsilon >0$.
\end{prop}
\begin{proof}
Following the proof of \cite[Theorem 4.6]{Evans2015} for $p=2$ until $(\star\star\star)$, we obtain 
\begin{align*}
\int_{\partial \Omega} |v|^2 \d \sigma \leq c_0 \int_\Omega |\nabla v||v|  \d x + c_0 \int_\Omega |v|^2  \d x 
\end{align*}
for all $v\in C^1(\Bar{\Omega})$. We obtain the desired estimate by Young's inequality for any $\epsilon>0$. We conclude by density of $C^1(\Bar{\Omega})$ in $H^1(\Omega)$ (see, e.g., \cite[Theorem 1.1]{Temam1977}). 
\end{proof}
We will apply this result for functions in $H^1_0(\Omega) \subset H^1(\Omega)$. We also need the following Sobolev embedding theorem (see, e.g., \cite[Theorem 4.12]{Adams2003}).
\begin{prop}\label{prop:sobolev_embedding}
Let $\Omega \subset \R^d$ be an open bounded domain with Lipschitz boundary. Then, the embedding $H^1_0(\Omega) \subset L^{p^\ast }(\Omega)$ is continuous, that is, there is a constant $c_0>0$ so that
\begin{align*}
\norm{v}_{L^{p^\ast}(\Omega)} \leq c_0\norm{v}_{H^1(\Omega)},
\end{align*}
where $p^\ast$ is the critical exponent given by 
\begin{align*}
p^\ast = \begin{cases}
\frac{2d}{d-2}, \qquad \qquad \qquad \quad d>2,\\
\text{any } q\in(1,\infty),  \quad \ \ \ \  d= 2.
\end{cases}
\end{align*}
\end{prop}
The next result can be found in \cite[Proposition 2.4]{Marino2019} or as an exercise in the book \cite[Exercise 4.1]{Troltzsch2010} for functions in $H^1(\Omega)$. However, it is clear from the proof in \cite[Proposition 2.4]{Marino2019} combined with Proposition \ref{prop:sobolev_embedding} that it holds for functions in $H^1_0(\Omega)$. 
\begin{prop}\label{prop:traceinfty}
Let $\Omega \subset \R^d$ be an open bounded domain and $\Gamma\subset \partial \Omega$ be Lipschitz. If $v \in H^1_0(\Omega) \cap L^\infty(\Omega)$, then $ v \in L^\infty(\Gamma)$ and 
\begin{align*}
\norm{ v}_{L^\infty(\Gamma)}   \leq c_0\norm{v}_{L^\infty(\Omega)}
\end{align*}
for  a constant  $c_0>0$.
\end{prop}

\subsection{Preliminary results}
We include some preliminary results needed to find uniform a priori estimates for the approximated solutions, two results usually used in the Moser iteration, and lastly, a result used to prove that the approximated problem converges to \eqref{eq:eq1}.
\begin{lem}\label{lemma:estn}
Let $X$ be a Banach space and $L^\infty_T X$ be as in Definition \ref{def:Bochner}. If $0<L<1$ and
\begin{align*}
\norm{v^n}_{L^\infty_T X} \leq  c\norm{v_0}_X + L \norm{v^{n-1}}_{L^\infty_T X}
\end{align*}
with $v^0:= v_0 \in X$ for any constant $c>0$, then
\begin{align*}
\norm{v^n}_{L^\infty_T X} \leq \Tilde{c}\norm{v_0}_X,
\end{align*}
where $\Tilde{c}$ is a positive constant independent of $n$.
\end{lem}
\begin{proof} 
We observe that
\begin{align*}
\norm{v^n}_{L^\infty_T X} &\leq c\norm{v_0}_X + L \norm{v^{n-1}}_{L^\infty_T X} \leq c(1+L)\norm{v_0}_X + L^2 \norm{v^{n-2}}_{L^\infty_T X}.
\end{align*}
By induction on $n\in \N$, it follows that
\begin{align*}
\norm{v^n}_{L^\infty_T X}  
&\leq c\norm{v_0}_X\sum_{k=0}^{n-1} L^k +L^n \norm{v_0}_{X}  \leq c\norm{v_0}_X\sum_{k=0}^{\infty} L^k.
\end{align*}
Since $0<L<1$, 
\begin{align*}
\norm{v^n}_{L^\infty_T X} &\leq \frac{c}{1-L}\norm{v_0}_X.
\end{align*}
Denoting $\Tilde{c}:= \frac{c}{1-L}$ concludes the proof.
\end{proof}
The next two results are commonly used in the application of the Moser iteration.
\begin{prop}\label{prop:nabla}
Let $v \in H^1(\Omega)$, then $v_+,v_- \in H^1(\Omega)$ and
\begin{align*}
\nabla v_+ &= \mathds{1}_{\{v> 0\}}
\nabla v 
&\text{and}&
&\nabla v_- = - \mathds{1}_{\{v< 0\}}
\nabla v  .
\end{align*}
Moreover,  if $v\in H^1_0(\Omega)$, then $v_+,v_- \in H^1_0(\Omega)$.
\end{prop}
\begin{lem}\label{lemma:passingthelimit_bound}
Let $\Omega \subset \R^d$ with $d\geq 2$ be a bounded domain with Lipschitz boundary $\Gamma \subset \partial \Omega$.  Let $v\in L^p(\Omega)$ with $v\geq 0$ and $1<p<\infty$ such that
\begin{align*}
\norm{v}_{L^{r_\ell}(\Omega)} \leq c
\end{align*}
for a universal constant $c>0$ and a sequence $\{r_\ell\}_\ell \subset (1,\infty)$ with  $r_\ell \rightarrow \infty$ as $\ell \rightarrow \infty$. Then $v\in L^\infty(\Omega)$.
\end{lem}
The proof of Proposition \ref{prop:nabla} can be found in \cite[Theorem 4.4]{Evans2015} and \cite[Lemma 3.1]{Drabek2022}, while the proof of Lemma \ref{lemma:passingthelimit_bound} can be found in, e.g., \cite[Proposition 2.2]{Marino2019} (see also the proof of \cite[Lemma 3.2]{Drabek1997}). The result underneath will be used to verify that the approximated solutions indeed solves \eqref{eq:eq1}.
\begin{lem}\label{lemma:strongtoae}
Let $X$ be a Banach space and $L^\infty_TX$ be as in Definition \ref{def:Bochner}. If $v^n \rightarrow v$ strongly in $L^\infty_TX$, then, up to a subsequence,
\begin{align*}
v^n(t) \rightarrow v(t) \text{ strongly in } X \text{ for a.e. } t\in (0,T).
\end{align*}
\end{lem}
The proof is similar to the first part of the proof in \cite[Lemma 13]{Zeng2018}, but we give the details for clarity and the sake of completeness.
\begin{proof}
With a proof of contraction, we claim that the sequence $\{v^n(t)\}_n \subset X$ has a subsequence that converges to $v(t) \in X$ for a.e. $t\in (0,T)$. If it does not hold, then there exists an interval $I\subset (0,T)$ with nonzero Lebesgue measure and a $\epsilon>0$ so that for all $N\in \N$, it implies $\norm{v^n(t)- v(t)}_X > \epsilon$ for all $n > N$ and $t\in I$. Consequently,
\begin{align*}
\underset{t\in(0,T)}{\mathrm{ess \: sup}}\norm{v^n(t)- v(t)}_X  \geq \underset{t\in I}{\mathrm{ess \: sup}} \norm{v^n(t) - v(t)}_X   >  \epsilon .
\end{align*}   
Thus, $	v^n(t) \rightarrow v(t)$  strongly in  $X$ (up to a subsequence) for a.e.  $t\in (0,T)$.
\end{proof}

\subsection{Auxiliary problem}
We also need to define an auxiliary problem to \eqref{eq:eq1a}, which will help us to define the iterative scheme in the proof of Theorem \ref{thm:mainresult}.
Let
\begin{subequations}
\begin{equation}\label{eq:assumptiona}
a : H^1_0(\Omega) \times H^1_0(\Omega) \rightarrow \R \text{ be bilinear, symmetric, coercive, and bounded,}
\end{equation} 
\begin{equation}\label{eq:assumption_psi}
\psi : H^1_0(\Omega) \rightarrow \R \text{ be proper, convex, and lower semicontinuous,}
\end{equation} 
and 
\begin{align}\label{eq:assumption_f}
f \in L^\infty_T H^1_0(\Omega).
\end{align}
\end{subequations}
We then consider the following problem.
\begin{prob}\label{prob:prelem}
Find $u \in L^\infty_T H^1_0(\Omega)$ such that
\begin{align*}
a(u(t),v-u(t)) + \psi(v) - \psi(u(t)) \geq \scalarprod{f(t)}{v-u(t)}_{H^1(\Omega)}
\end{align*}
for all $v\in H^1_0(\Omega)$ and a.e. $t\in (0,T)$.
\end{prob}
In \cite[Theorem 3.12]{Sofonea2009}, they proved the following existence and uniqueness result for Problem \ref{prob:prelem}.
\begin{thm}\label{thm:pre_existence_uniqueness}
Assume that \eqref{eq:assumptiona}-\eqref{eq:assumption_f} hold. Then, Problem \ref{prob:prelem} has a unique solution $u\in L^\infty_T H^1_0(\Omega)$.
\end{thm}

\section{Main result}\label{sec:mainresult}
In this section, we state our main result. First, we present the conditions we wish to investigate \eqref{eq:eq1}. That is, the structure on $\varphi$, $\phi$, and $H$. These assumptions are inspired by new applications (see Appendix \ref{appendix:application}).
\\
\noindent
$(H1)$: \label{assumptionvarphi}  $\varphi : \Gamma_C  \times \R \times \R \times \R \rightarrow \R \text{ is such that }$
\begin{itemize}  
\item[$(i)$] $x\mapsto \varphi(x,y,r,v) \text{ is measurable in } \Gamma_C$ for all $y,r,v \in \R$.
\label{list:varphi1}
\item[$(ii)$] $\text{There is $c_{j\varphi} \geq 0$ for $j=1,2$ such that }$ 
\begin{align*}
& \varphi (x,y_1,r_1, v_2) - \varphi (x,y_1,r_1, v_1) +\varphi (x,y_2,r_2,v_1)  - \varphi (x,y_2,r_2, v_2) \\
&\leq c_{1\varphi} |y_1-y_2||v_1-v_2| +c_{2\varphi} |r_1-r_2| |v_1-v_2| 
\end{align*}
$\text{for all } y_i,z_i,r_i,v_i\in \R$ with $i=1,2$ and a.e. $x\in \Gamma_C$.
\item[$(iii)$] There is a $c_{0\varphi} \geq 0$ so that $|\varphi (x,0,0, v_1) - \varphi (x,0,0, v_2)| \leq c_{0\varphi} |v_1-v_2|$ $\text{for all } v_1,v_2 \in \R$ and a.e. $x\in \Gamma_C$.
\label{list:varphi2}
\end{itemize}
$(H2)$: \label{assumptiophi}  $\phi : \Gamma_C  \times \R \times \R \times \R \rightarrow \R \text{ is such that }$
\begin{itemize}  
\item[$(i)$] $x\mapsto \phi(x,y,r,v) \text{ is measurable in } \Gamma_C$ for all $y,r,v \in \R$.
\label{list:phi1}
\item[$(ii)$] $\text{There is $c_{j\phi} \geq 0$ for $j=1,2,3$ such that }$ 
\begin{align*}
& \phi (x,y_1,r_1, v_2) - \phi (x,y_1,r_1, v_1) +\phi (x,y_2,r_2,v_1)  - \phi (x,y_2,r_2, v_2) \\
&\leq c_{1\phi} |y_1|^2|r_1-r_2| |v_1-v_2|  + c_{2\phi}|y_2||r_2||y_1-y_2||v_1-v_2| + c_{3\phi}|y_1||r_2||y_1-y_2||v_1-v_2|
\end{align*}
$\text{for all } y_i,z_i,r_i,v_i\in \R$ with $i=1,2$ and a.e. $x\in \Gamma_C$.
\item[$(iii)$]  $\phi (x,0,0, v) =0$ $\text{for all } v \in \R$ and a.e. $x\in \Gamma_C$.
\label{list:phi2}
\end{itemize}
$(H3)$: \label{assumptionH}  $ H : \Gamma_C \times \mathbb{R} \times \mathbb{R} \rightarrow \mathbb{R} \text{ is such that }$
\begin{itemize}
\item[$(i)$] The mapping $x\mapsto  H(x,y,r)$ is measurable on $\Gamma_C$ for all $y,r \in \mathbb{R}$.
\item[$(ii)$] There exists $c_{j \beta}\geq 0$ for $j=1,2,3$ such that $(y,r) \mapsto   H(x,y,r)$ satisfies 
\begin{align*}
&|  H(x,y_1,r_1)  -   H (x,y_2,r_2)|\leq  
( c_{1 \beta}|y_1||r_1|+c_{2 \beta}|y_1||r_2|)|r_1-r_2|  +  c_{3\beta}|r_2|^2|y_1-y_2| 
\end{align*}
for all $y_i,r_i \in \mathbb{R}$ for $i=1,2$, a.e. $x\in \Gamma_C$.
\item[$(iii)$] There is a $c_{0\beta}\geq 0$ so that $ H (x,0,0) =c_{0\beta}$ for a.e. $x\in \Gamma_C$.\label{list:H_0}
\end{itemize}
Moreover, we assume that
\begin{subequations}\label{eq:assumptions_mu}
\begin{align}\label{eq:assumptionmu}
\mu_\ast <\mu(x) \text{ for a.e. } x\in \Omega \text{ and }  \mu\in L^\infty(\Omega)
\end{align}
with the smallness assumption
\begin{align}\label{eq:smallness}
\mu_\ast > c_{2\varphi}c_0 +  c_{1\phi}c_0\norm{\beta_0}^2_{L^\infty(\Gamma_C)}.
\end{align}
\end{subequations}
We consider the following regularity on the data 
\begin{align}\label{eq:assumptionondata}
f_0 &\in  L^\infty_TL^\infty(\Omega), & f_N &\in L^\infty_TL^\infty(\Gamma_N), &\text{ and }  & &\beta_0 &\in  L^\infty(\Gamma_C).
\end{align}
\begin{remark}
In practice, \eqref{eq:smallness} is a condition on the elasticity of the material. 
\end{remark}
Before stating the main result, we will do some calculations that will become useful in the proof.
\begin{lem}\label{lemma:calculation}
Under the assumption \hyperref[assumptionvarphi]{$(H1)$}-\hyperref[assumptionH]{$(H3)$}, we have
\begin{align*}
\varphi(x,y, r, v_2)-\varphi(x,y,r, v_1)  &\leq c_{0\varphi} |v_1-v_2| + c_{1\varphi} |y||v_1-v_2| +c_{2\varphi} |r| |v_1-v_2|,\\
\phi(x,y, r, v_2)-\phi(x,y,r, v_1)  &\leq  c_{1\phi} |y|^2|r||v_1-v_2|,\\
|H(x,y, r)|  &\leq c_{0\beta} + c_{3\beta} |y||r|^2
\end{align*}
for all $y,r\in \R$, $v_1,v_2\in \R$, and a.e. $x\in \Gamma_C$. 
\end{lem}
\begin{proof}
We observe by \hyperref[assumptionvarphi]{$(H1)$}$(ii)$ that
\begin{align*}
&\varphi(x,y, r, v_2)-\varphi(x,y,r, v_1) + \varphi(x,0,0,v_1) - \varphi(x,0,0,v_2)   \\
&\leq  c_{1\varphi} |y||v_1-v_2| +c_{2\varphi} |r| |v_1-v_2|
\end{align*}
for all $y,r\in \R$, $v_1,v_2\in \R$ and a.e. $x\in \Gamma_C$. Rearranging and utilizing \hyperref[assumptionvarphi]{$(H1)$}$(iii)$ yields
\begin{align*}
&\varphi(x,y, r, v_2)-\varphi(x,y,r, v_1)     \\
&\leq c_{1\varphi} |y||v_1-v_2| +c_{2\varphi} |r| |v_1-v_2| + |\varphi(x,0, 0, v_1)-\varphi(x,0,0, v_2)|  \notag\\
&\leq c_{0\varphi} |v_1-v_2| + c_{1\varphi} |y||v_1-v_2| +c_{2\varphi} |r| |v_1-v_2|
\end{align*}
for all $y,r\in \R$, $v_1,v_2\in \R$ and a.e. $x\in \Gamma_C$. Next, from \hyperref[assumptiophi]{$(H2)$}$(ii)$-$(iii)$, we have
\begin{align*}
\phi (x,y,r,v_2)  - \phi (x,y,r, v_1)
&= \phi (x,y,r,v_2)  - \phi (x,y,r, v_1) + \phi (x,0,0, v_1) - \phi (x,0,0, v_2)  \\
&\leq c_{1\phi} |y|^2|r| |v_1-v_2| 
\end{align*}
for all $y,r\in \R$, $v_1,v_2\in \R$ and a.e. $x\in \Gamma_C$. Lastly, it follows by \hyperref[assumptionH]{$(H3)$}$(ii)$ that
\begin{align*}
| H (x,0,0) - H(x,y,r) | &\leq  c_{3\beta}|y||r|^2,
\end{align*}
which we combine with
\begin{align*}
| H(x,y,r)| &\leq |  H(x,y,r)  -   H (x,0,0)| +  |H(x,0,0)|
\end{align*}
and \hyperref[assumptionH]{$(H3)$}$(iii)$ to obtain the desired estimate.
\end{proof}
We will now state the main result, i.e., Theorem \ref{thm:mainresult}, which provides existence, uniqueness, and regularity results for \eqref{eq:eq1}.
\begin{thm}\label{thm:mainresult}
Assume that \hyperref[assumptionvarphi]{$(H1)$}-\hyperref[assumptionH]{$(H3)$} and \eqref{eq:assumptions_mu}-\eqref{eq:assumptionondata} hold. 
Then there exists a $T>0$ so that $(u,\beta) \in  L^\infty_T(H^1_0(\Omega) \cap L^\infty(\Omega)) \times C([0,T]; L^2(\Gamma_C)) \cap  L^\infty_T L^\infty(\Gamma_C)$ is the unique solution to \eqref{eq:eq1}.
\end{thm}
\begin{remark}\label{remark:mainthm}
    In the proof, we combine techniques from \cite{Drabek1997,Gilbarg2001} (see also \cite{Drabek2022}) to obtain the boundedness of the solution. For this, we need to adjust the Moser iteration scheme to our system of equations (see Remark \ref{remark:moser} for more details).
\end{remark}
\begin{remark}
If $\phi \equiv 0$, i.e., $c_{1\phi}=c_{2\phi}=c_{3\phi}=0$ in \hyperref[assumptiophi]{$(H2)$}, the smallness assumption \eqref{eq:smallness} is also found in, e.g., \cite{Sofonea2002,Sofonea2020,Migorski2010}. 
\end{remark}

\subsection{Strategy of the proof of Theorem \ref{thm:mainresult}}
The proof of the theorem is divided into six steps. In the first step, we introduce an auxiliary problem to \eqref{eq:eq1} (Problem \ref{prob:aux}).
In the auxiliary problem, we fix two of the functions in \eqref{eq:eq1a} and leave \eqref{eq:eq1b} intact. We use an existing result to prove that Problem \ref{prob:aux} has a unique solution (Step 2). We then find an estimate and boundedness for the solution to the auxiliary problem (Step 3-4). We then define an iterative scheme for Problem  \ref{prob:aux} using \eqref{eq:eq1}. This iterative scheme decouples \eqref{eq:eq1b} and Problem \ref{prob:aux} at each step. We then show a priori estimates and boundedness results for the sequential solution using the estimates found in Step 3-4. Next, we study the difference between two successive iterates and show that these iterates are Cauchy sequences. We then pass to the limit to show that the iterative scheme converges to \eqref{eq:eq1} (Step 5). In Step 6, we show that the solution to \eqref{eq:eq1} is unique.

\section{Proof of Theorem \ref{thm:mainresult}}\label{sec:proof_of_mainresult}
With the preparation in Section \ref{sec:prelim}-\ref{sec:mainresult}, we proceed to the proof of Theorem \ref{thm:mainresult}. For the convenience of the reader, we split the proof into several steps. 

\subsection*{Step 1 \textit{(Auxiliary problem to the elliptic variational inequality \eqref{eq:eq1a})}.}
Let $(\xi,\beta) \in L^\infty_T (H^1_0(\Omega) \cap L^\infty(\Omega)) \times C([0,T];L^2(\Gamma_C)) \cap  L^\infty_T  L^\infty(\Gamma_C)$ be given. We then define an auxiliary problem to \eqref{eq:eq1a}.
\begin{prob}\label{prob:aux}
Find $u_{\xi\beta}\in L^\infty_T H^1_0(\Omega)$  corresponding to $(\xi,\beta) \in L^\infty_T (H^1_0(\Omega) \cap L^\infty(\Omega)) \times C([0,T];L^2(\Gamma_C)) \cap  L^\infty_T  L^\infty(\Gamma_C)$ such that
\begin{align*}
\int_{\Omega} \mu \nabla u_{\xi\beta}(t)\cdot \nabla(v-u_{\xi\beta}(t)) \d x  &+ \int_{\Gamma_C} [\varphi(\beta(t),  \xi(t),  v) - \varphi(\beta(t),  \xi(t),  u_{\xi\beta}(t))] \d \sigma \\
&+ \int_{\Gamma_C} [\phi(\beta(t),  \xi(t),  v) - \phi(\beta(t),  \xi(t),  u_{\xi\beta}(t))] \d \sigma\\
&\geq \int_{\Omega} f_0(t)  (v-u_{\xi\beta}(t)) \d x+ \int_{\Gamma_N} f_N(t) (v-u_{\xi\beta}(t)) \d \sigma
\end{align*}
for all $v\in H^1_0(\Omega)$ and a.e. $t\in (0,T)$.
\end{prob}
\begin{remark}
In comparison to \eqref{eq:eq1}, we keep $\xi = u$ and $\beta$ (still denoted by $\beta$) known. The subscript on $u$ is to emphasize that the solution to Problem \ref{prob:aux} corresponds to  $(\xi,\beta) \in L^\infty_T (H^1_0(\Omega) \cap L^\infty(\Omega)) \times C([0,T];L^2(\Gamma_C)) \cap  L^\infty_T  L^\infty(\Gamma_C)$.
\end{remark}

\subsection*{Step 2 \textit{(Existence and uniqueness of a solution to the auxiliary problem (Problem \ref{prob:aux}))}.} 
Let $(\xi,\beta) \in L^\infty_T (H^1_0(\Omega) \cap L^\infty(\Omega)) \times C([0,T];L^2(\Gamma_C)) \cap  L^\infty_T  L^\infty(\Gamma_C)$ be given. We define the bilinear operator $a : H^1_0(\Omega) \times H^1_0(\Omega) \rightarrow \R$ by
\begin{align*}
a(u,v) =\int_\Omega \mu(x) \nabla u \cdot \nabla v \d x
\end{align*}
for all $u,v\in H^1_0(\Omega)$ and the functional $\psi_{\xi\beta} : H^1_0(\Omega) \rightarrow \R$ by
\begin{align*}
\psi_{\xi\beta}(v) = \int_{\Gamma_C} \varphi(\beta,  \xi,  v) \d \sigma +  \int_{\Gamma_C} \phi(\beta,  \xi,  v) \d \sigma
\end{align*}
for all $\xi,v \in H^1_0(\Omega)$ and $\beta \in L^2(\Gamma_C)$. Moreover, by the Riesz representation theorem, there exists a unique $f :[0,T] \rightarrow H^1_0(\Omega)$ so that
\begin{align*}
\scalarprod{f(t)}{v}_{H^1(\Omega)}  =  \int_{\Omega} f_0(t)  v \d x+ \int_{\Gamma_N} f_N(t) v \d \sigma
\end{align*}
for all $v\in H^1_0(\Omega)$ and a.e. $t\in (0,T)$. From \hyperref[assumptionvarphi]{$(H1)$}-\hyperref[assumptiophi]{$(H2)$}, Lemma \ref{lemma:calculation}, and \eqref{eq:assumptions_mu}-\eqref{eq:assumptionondata} it follows directly by Theorem \ref{thm:pre_existence_uniqueness} that Problem \ref{prob:aux} has a unique solution $u_{\xi\beta} \in L^\infty_T H^1_0(\Omega)$ corresponding to  $(\xi,\beta) \in L^\infty_T (H^1_0(\Omega) \cap L^\infty(\Omega)) \times C([0,T];L^2(\Gamma_C)) \cap  L^\infty_T  L^\infty(\Gamma_C)$. 

\subsection*{Step 3 \textit{(Estimate on the solution to the auxiliary problem (Problem \ref{prob:aux}))}.}
Indeed, we will now find an estimate on the solution to Problem \ref{prob:aux} that will come in handy later.
\begin{prop}\label{prop:estn}
Assume that \hyperref[assumptionvarphi]{$(H1)$}-\hyperref[assumptiophi]{$(H2)$} and \eqref{eq:assumptions_mu}-\eqref{eq:assumptionondata} hold. For given $(\xi,\beta) \in L^\infty_T (H^1_0(\Omega) \cap L^\infty(\Omega)) \times C([0,T];L^2(\Gamma_C)) \cap  L^\infty_T  L^\infty(\Gamma_C)$, let $u_{\xi\beta}$ be the solution to Problem \ref{prob:aux}. Then there exists a universal constant $c>0$ such that
\begin{align*}
\norm{u_{\xi\beta}}_{L^\infty_T H^1(\Omega)}  \leq& c(1 +   \norm{f_0}_{L^\infty_T L^\infty(\Omega)} + \norm{f_N}_{L^\infty_T L^\infty(\Gamma_N)}) +\frac{c_{1\varphi} c_0}{\mu_\ast} \norm{\beta}_{L^\infty_TL^\infty(\Gamma_C)}\\
&+ \frac{c_{2\varphi} c_0}{\mu_\ast} \norm{\xi}_{L^\infty_T H^1(\Omega)} +  \frac{c_{1\phi} c_0}{\mu_\ast}\norm{\beta}_{L^\infty_TL^\infty(\Gamma_C)}^2 \norm{\xi}_{L^\infty_T H^1(\Omega)}.
\end{align*}
\end{prop}
\begin{remark}
The estimate in Proposition \ref{prop:estn} will be used to find estimates of the solutions of the iterative scheme. The right-hand side will therefore depend on $n\in\N$. But with this explicit dependency, we will be able to prove uniform estimates.
\end{remark}
\begin{proof}
For simplicity of notation, we denote $u=u_{\xi\beta}$.	Choosing $v=0$ in Problem \ref{prob:aux} gives us 
\begin{align*}
\int_\Omega \mu |\nabla u(t)|^2  \d x &\leq \int_{\Omega} f_0(t)   u(t) \d x + \int_{\Gamma_N} f_N(t)    u(t) \d \sigma\\
&+\int_{\Gamma_C} [-\varphi(\beta(t),  \xi(t), 0)+\varphi(\beta(t), \xi(t), u(t))] \d \sigma\\
&+\int_{\Gamma_C} [-\phi(\beta(t),  \xi(t), 0)+\phi(\beta(t), \xi(t), u(t))] \d \sigma
\end{align*}
for a.e. $t\in (0,T)$.
Utilizing Lemma \ref{lemma:calculation} with $v_1=u(t)$ and $v_2 = 0$, \eqref{eq:assumptionmu}, H\"{o}lder's inequality, and Proposition \ref{prop:trace} with $\epsilon = 1$ yields
\begin{align*}
\mu_\ast \int_\Omega |\nabla u(t)|^2 \d x &\leq 
\norm{f_0(t)}_{L^\infty(\Omega)} \int_\Omega |u(t)| \d x + \norm{f_N(t)}_{L^\infty(\Gamma_N)} \int_{\Gamma_N} | u(t)| \d \sigma \\
&+ (c_{0\varphi} + c_{1\varphi}\norm{\beta(t)}_{L^\infty(\Gamma_C)})\int_{\Gamma_C}   | u(t)| \d \sigma \\
&+(c_{2\varphi}+ c_{1\phi} \norm{\beta(t)}_{L^\infty(\Gamma_C)}^2  ) \int_{\Gamma_C} |\xi(t)|  | u(t)| \d \sigma  \\ \notag 
&\leq c_0 \norm{f_0(t)}_{L^\infty(\Omega)} \norm{u(t)}_{L^2(\Omega)} + c_0\norm{f_N(t)}_{L^\infty(\Gamma_N)}\norm{u(t)}_{H^1(\Omega)}\\ \notag
&+\Big(c_{0\varphi}c_0+ c_{1\varphi}c_0\norm{\beta(t)}_{L^\infty(\Gamma_C)}\Big) \norm{u(t)}_{H^1(\Omega)}\\
&+ c_0\Big(c_{2\varphi}+c_{1\phi}\norm{\beta(t)}_{L^\infty(\Gamma_C)}^2\Big)\norm{\xi(t)}_{H^1(\Omega)} \norm{u(t)}_{H^1(\Omega)}
\notag
\end{align*}
for a.e. $t\in (0,T)$. From the Poincar\'{e} inequality, we obtain
\begin{align*}
\mu_\ast\norm{u}_{L^\infty_T H^1(\Omega)}^2 &\leq c\Big(1 +  \norm{f_0}_{L^\infty_T  L^\infty(\Omega)} + \norm{f_N}_{L^\infty_T L^\infty(\Gamma_N)} \Big)\norm{u}_{L^\infty_T H^1(\Omega)} \\
&+c_{1\varphi}c_0\norm{\beta}_{L^\infty_TL^\infty(\Gamma_C)}\norm{u}_{L^\infty_T H^1(\Omega)}\\
&+ c_0\Big(c_{2\varphi}+c_{1\phi}\norm{\beta}_{L^\infty_TL^\infty(\Gamma_C)}^2\Big)\norm{\xi}_{L^\infty_T H^1(\Omega)} \norm{u}_{L^\infty_T H^1(\Omega)} ,
\end{align*}
which implies
\begin{align*}
\norm{u}_{L^\infty_T H^1(\Omega)}     &\leq c(1 +   \norm{f_0}_{L^\infty_T L^\infty(\Omega)} + \norm{f_N}_{L^\infty_T L^\infty(\Gamma_N)}) \\
&+\frac{c_{1\varphi} c_0}{\mu_\ast} \norm{\beta}_{L^\infty_TL^\infty(\Gamma_C)} + \frac{c_{2\varphi}c_0}{\mu_\ast} \norm{\xi}_{L^\infty_T H^1(\Omega)}  + \frac{c_{1\phi}c_0}{\mu_\ast} \norm{\beta}_{L^\infty_TL^\infty(\Gamma_C)}^2\norm{\xi}_{L^\infty_T H^1(\Omega)} 
\end{align*}
as desired.
\end{proof}

\subsection*{Step 4 \textit{(Boundedness of the solution to the auxiliary problem)}.}
Let $(\xi,\beta)\in L^\infty_T (H^1_0(\Omega) \cap L^\infty(\Omega)) \times C([0,T];L^2(\Gamma_C)) \cap  L^\infty_TL^\infty(\Gamma_C)$ be known. Indeed, we summarize the boundedness of $u_{\xi\beta}$ corresponding to  $(\xi,\beta) \in L^\infty_T (H^1_0(\Omega) \cap L^\infty(\Omega)) \times C([0,T];L^2(\Gamma_C))\cap L^\infty_TL^\infty(\Gamma_C)$ in the next proposition.
\begin{prop}\label{prop:u_bound}
Assume that \hyperref[assumptionvarphi]{$(H1)$}-\hyperref[assumptiophi]{$(H2)$} and \eqref{eq:assumptions_mu}-\eqref{eq:assumptionondata} hold. Let $(\xi,\beta) \in L^\infty_T (H^1_0(\Omega) \cap L^\infty(\Omega)) \times C([0,T];L^2(\Gamma_C)) \cap L^\infty_TL^\infty(\Gamma_C)$ be given and let $u_{\xi\beta} \in L^\infty_T H^1_0(\Omega)$ be the unique solution to Problem \ref{prob:aux}. Then, there is a constant $c_0>0$ such that
\begin{align}\label{eq:u_xibeta_Linfty}
\norm{u_{\xi\beta}}_{L^\infty_T L^\infty(\Omega)} &\leq c_0K,
\end{align}
where 
\begin{align*}
K=&\frac{1}{\mu_\ast}\Big(c_{0\varphi}+\norm{f_0}_{L^\infty_TL^\infty(\Omega)} +\norm{f_N}_{L^\infty_TL^\infty(\Gamma_N)}  \\
&+ c_{1\varphi}\norm{\beta}_{L^\infty_TL^\infty(\Gamma_C)} + c_{2\varphi} c_0\norm{\xi}_{L^\infty_TL^\infty(\Omega)} + c_{1\phi}c_0\norm{\beta}_{L^\infty_TL^\infty(\Gamma_C)}^2 \norm{\xi}_{L^\infty_TL^\infty(\Omega)} \Big).
\end{align*}
\end{prop}
\begin{remark}
The estimate in Proposition \ref{prop:u_bound} will be used in the iterative scheme, where $K$ will depend on $n$. But, we know the explicit dependence which will make us able to prove that the estimates are uniform with respect to $n$. 
\end{remark}
\begin{remark}\label{remark:moser}
The approach to the proof of Proposition \ref{prop:u_bound} is based on the work  of \cite[Theorem 8.16]{Gilbarg2001} and \cite[Lemma 3.2]{Drabek1997}  (see also \cite{Drabek2022}). Due to the coupling of the system, the proof is different compared to the aforementioned results. In particular, one of the technical difficulties arises when obtaining \eqref{eq:u_xibeta_Linfty}, where the constant $K$ depends on $\norm{\xi(t)}_{L^\infty(\Omega)}$. In contrary to \cite{Drabek2022} (see also \cite[Lemma 3.2]{Drabek1997}), where it suffices to have $\norm{u_{\xi\beta}(t)}_{L^2(\Omega)}$ in the upper bound. This is a key ingredient to close our estimates (see Step 5). 
\end{remark}
\begin{remark}
    The scheme of our proof can be adapted to (but not limited to) other coupled systems of equations.
\end{remark}
\begin{proof}
To simplify the notation, we denote $u=u_{\xi\beta}$. We may write $u = u_+ - u_-$ with $u_+ = \max\{u,0\}$ and $u_- = \max\{-u,0\}$. As a consequence, if we prove that $\norm{u_+}_{L^\infty_TL^\infty(\Bar{\Omega})} \leq c_0K$ and  $\norm{u_-}_{L^\infty_TL^\infty(\Bar{\Omega})} \leq c_0K$, then 
\begin{align*}
\norm{u}_{L^\infty_TL^\infty(\Bar{\Omega})} \leq \norm{u_+ - u_-}_{L^\infty_TL^\infty(\Bar{\Omega})} \leq \norm{u_+}_{L^\infty_TL^\infty(\Bar{\Omega})}+ \norm{u_-}_{L^\infty_TL^\infty(\Bar{\Omega})} \leq 2 c_0K.
\end{align*}
\subsection*{Case 1: Estimating $\norm{u_+}_{L^\infty_TL^\infty(\Bar{\Omega})} \leq c_0 K$} For convenience, we divided the proof into four main parts. First, we define
\begin{align}\label{eq:def_w}
w(t) := \log \Big( \frac{M(t)+K}{M(t)+K-u_+(t)}\Big) = \log(M(t)+K) - \log(M(t)+K-u_+(t))
\end{align}
where $M(t):= \norm{u_+(t)}_{L^\infty(\Bar{\Omega})}$ for a.e. $t\in(0,T)$.  We will use $w$ in \eqref{eq:def_w} to obtain the desired estimate. The proof will be as follows; first, we show that $\norm{w(t)}_{H^1(\Omega)} \leq c_0$. Secondly, we verify that $w(t)$ is nonnegative. Thirdly, we will obtain the estimate $\norm{w(t)}_{L^\infty(\Bar{\Omega})} \leq c_0$, and lastly, we conclude by using the three previous steps.
\subsection*{Part I: $\norm{w(t)}_ {H^1(\Omega)} \leq c_0$} 
It is important in this step that the constant is independent of $K$. It will be made clear in Part IV, but in short, we will use the exponential in \eqref{eq:def_w} to obtain the desired estimate. To show the estimate $\norm{w(t)}_{H^1(\Omega)} \leq c_0$, we will use Problem \ref{prob:aux}. First, we have to find a convenient choice of $v \in H^1_0(\Omega)$. That is, choosing a test function $v$ such that we can rewrite Problem \ref{prob:aux} in terms of $w$. From Proposition \ref{prop:nabla}, we have
\begin{align*}
	\nabla w(t) = \frac{\nabla u_+(t)}{M(t)+K-u_+(t)}.
\end{align*}
Therefore, a suitable choice for the test function $v$ is $v = u(t)- \frac{1}{M(t)+K-u_+(t)} \in H^1_0(\Omega)$ in Problem \ref{prob:aux}.
Indeed, since $u_+(t) \in H^1_0(\Omega)$ (by Proposition \ref{prop:nabla}) and $K \leq (M(t)-u_+(t))+K$ (from $u_+(t) \leq M(t)$), we have
\begin{align*}
    \nabla v = \nabla u(t) - \frac{\nabla u_+(t)}{(M(t)+K-u_+(t))^2}
\end{align*}
and
\begin{align*}
    \frac{1}{M(t)+K-u_+(t)} \leq \frac{1}{K}.
\end{align*}
This yields $v\in H^1_0(\Omega)$. Then choosing $v =u(t)- \frac{1}{M(t)+K-u_+(t)}$, Problem \ref{prob:aux} becomes
\begin{align*}
\int_\Omega \mu &\frac{\nabla u_+(t) \cdot \nabla u_+(t)}{(M(t)+K-u_+(t))^2} \d x \\
&\leq \int_\Omega \frac{|f_0(t)|}{M(t)+K-u_+(t)} \d x  + \int_{\Gamma_N} \frac{|f_N(t)|}{M(t)+K-u_+(t)} \d \sigma \\
&+  \int_{\Gamma_C}  [\varphi(\beta(t),u(t),u(t)-\frac{1}{M(t)+K-u_+(t)}) -\varphi(\beta(t),u(t),u(t)) ]  \d \sigma\\
&+  \int_{\Gamma_C}  [\phi(\beta(t),u(t),u(t)-\frac{1}{M(t)+K-u_+(t)}) -\phi(\beta(t),u(t),u(t)) ]  \d \sigma
\end{align*}
for a.e. $t\in (0,T)$. From Lemma \ref{lemma:calculation} and the lower bound of $\mu$ \eqref{eq:assumptionmu}, we deduce
\begin{align*}
\mu_\ast \int_\Omega \frac{|\nabla u_+(t)|^2}{(M(t)+K-u_+(t))^2} \d x &\leq \int_\Omega \frac{|f_0(t)|}{M(t)+K-u_+(t)}  \d x + \int_{\Gamma_N} \frac{|f_N(t)|}{M(t)+K-u_+(t)} \d \sigma \\
&+  \int_{\Gamma_C}  \frac{c_{0\varphi} + c_{1\varphi}|\beta(t)| + c_{2\varphi}|\xi(t)|}{M(t)+K-u_+(t)}   \d \sigma +  \int_{\Gamma_C}  \frac{c_{1\phi} |\beta(t)|^2|\xi(t)|}{M(t)+K-u_+(t)}  \d \sigma
\end{align*}
for a.e. $t\in(0,T)$. With $\frac{1}{M(t)+K-u_+(t)}\leq \frac{1}{K}$, we insert the definition of $w$ \eqref{eq:def_w} to obtain
\begin{align*}
\int_\Omega |\nabla w(t)|^2 \d x &\leq \frac{1}{\mu_\ast K}\int_\Omega |f_0(t)|  \d x + \frac{1}{\mu_\ast K}\int_{\Gamma_N} |f_N(t)| \d \sigma \\
&+  \frac{1}{\mu_\ast K}\int_{\Gamma_C} (c_{0\varphi} + c_{1\varphi}|\beta(t)| + c_{2\varphi}|\xi(t)|)  \d \sigma\\
&+  \frac{1}{\mu_\ast K}\int_{\Gamma_C}  c_{1\phi}|\beta(t)|^2|\xi(t)|  \d \sigma
\end{align*}
for a.e. $t\in (0,T)$. Observe, by H\"{o}lder's inequality that
\begin{align*}
\int_\Omega  |\nabla w(t)|^2 \d x  &\leq \frac{c_0}{\mu_\ast K} \Big( c_{0\varphi}+ \norm{f_0(t)}_{L^\infty(\Omega)} +\norm{f_N(t)}_{L^\infty(\Gamma_N)} + c_{1\varphi}\norm{\beta(t)}_{L^\infty(\Gamma_C)} \\
&+ c_{2\varphi} c_0\norm{\xi(t)}_{L^\infty(\Omega)} + c_{1\phi}c_0\norm{\beta(t)}_{L^\infty(\Gamma_C)}^2\norm{\xi(t)}_{L^\infty(\Omega)} \Big)
\end{align*}
for a.e. $t \in (0,T)$.	By the definition of $K$, we get 
\begin{align*}
\int_\Omega  |\nabla w(t)|^2 \d x  \leq c_0
\end{align*}
for a.e. $t\in (0,T)$. With an application of the Poincar\'{e} inequality,  we obtain the desired estimate.
\subsection*{Part II: $w(t)$ is nonnegative} Indeed, since  $u_+(t) \geq 0$, we have
\begin{align*}
    M(t)+K \geq M(t)+K-u_+(t),
\end{align*}
which implies
\begin{align*}
    \frac{M(t)+K}{M(t)+K-u_+(t)} \geq 1.
\end{align*}
Consequently, by the definition of $w$ \eqref{eq:def_w}, it follows that  $w(t)\geq 0$ for a.e. $t\in (0,T)$.

\subsection*{Part III: $\norm{w(t)}_{L^\infty(\Bar{\Omega})} \leq c_0$} 
We can divide this part into two steps since
\begin{align*}
    \norm{w(t)}_{L^\infty(\Bar{\Omega})}  &\leq \norm{w(t)}_{L^\infty(\Omega)}  + \norm{w(t)}_{L^\infty(\partial\Omega)}. 
\intertext{In the first step, we will show the estimate}
	 \norm{w(t)}_{L^\infty(\Omega)} &\leq c_0,
\intertext{and as a consequence of the first step, we obtain}
	\norm{w(t)}_{L^\infty(\partial \Omega)} &\leq \norm{w(t)}_{L^\infty(\Gamma_N)}  + \norm{w(t)}_{L^\infty(\Gamma_C)}\leq c_0
\end{align*}
in the second step for a.e. $t\in(0,T)$.
\subsection*{Part III.1: $\norm{w(t)}_{L^\infty(\Omega)} \leq c_0$} 
The proof of this part relies on the Moser iteration. The goal is to obtain an estimate where the constant is independent of $K$. To this extent, we first go back to Problem \ref{prob:aux} to choose another test function $v\in H^1_0(\Omega)$. The choice of $v$ has to be done in a way so that we may use the definition of $w$ \eqref{eq:def_w} to rewrite Problem \ref{prob:aux} in terms of $w$ instead of $u$. At the same time, we need the equation to be independent of the functions in $K$.  
\\
\indent
For the sake of simplicity, the dependence on $t$ in $w$ will be omitted for the rest of this part. We define $w_h = \min\{w, h\}$ for any $h >0$ and choose $v=u(t)- \frac{w_h^{2\kappa} w}{M(t)+K-u_+(t)} \in H^1_0(\Omega)$ for any $\kappa>0$, where it follows by Proposition \ref{prop:nabla} that
\begin{align*}
\nabla v = \nabla u(t) -   \frac{\nabla (w_h^{2\kappa} w)}{M(t)+K-u_+(t)}  - \frac{w_h^{2\kappa} w\nabla u_+(t)}{(M(t)+K-u_+(t))^2}.
\end{align*}
We note here that $w_h \geq 0$ since $w\geq 0$ from Part II. Using this test function in Problem \ref{prob:aux} yields
\begin{align*}
\int_\Omega &\mu \frac{ \nabla u(t) \cdot \nabla (w_h^{2\kappa} w) }{M(t)+K-u_+(t)} \d x + \int_\Omega \mu   \frac{w_h^{2\kappa} w |\nabla u_+(t)|^2 }{(M(t)+K-u_+(t))^2} \d x \\
&\leq \int_\Omega \frac{|f_0(t)|  w_h^{2\kappa} w}{M(t)+K-u_+(t)} \d x + \int_{\Gamma_N} \frac{|f_N(t)|w_h^{2\kappa} w}{M(t)+K-u_+(t)}  \d \sigma \\
&+  \int_{\Gamma_C} [\varphi(\beta(t),u(t),u(t)- \frac{w_h^{2\kappa} w}{M(t)+K-u_+(t)})-\varphi(\beta(t),u(t),u(t))]  \d \sigma\\
&+  \int_{\Gamma_C} [\phi(\beta(t),u(t),u(t)- \frac{w_h^{2\kappa} w}{M(t)+K-u_+(t)})-\phi(\beta(t),u(t),u(t))]  \d \sigma
\end{align*}
for a.e. $t\in(0,T)$. Utilizing Lemma \ref{lemma:calculation} reads
\begin{align*}
\int_\Omega \mu \frac{ \nabla u(t) \cdot \nabla (w_h^{2\kappa} w) }{M(t)+K-u_+(t)} \d x + \int_\Omega &\mu   \frac{w_h^{2\kappa} w |\nabla u_+(t)|^2 }{(M(t)+K-u_+(t))^2} \d  x \\
&\leq \int_\Omega \frac{|f_0(t)|w_h^{2\kappa} w}{M(t)+K-u_+(t)}  \d x \\
&+ \int_{\Gamma_N} \frac{|f_N(t)| w_h^{2\kappa} w}{M(t)+K-u_+(t)}  \d \sigma \\
&+  \int_{\Gamma_C}  \frac{c_{0\varphi} + c_{1\varphi}|\beta(t)| + c_{2\varphi}|\xi(t)|}{M(t)+K-u_+(t)} w_h^{2\kappa} w   \d \sigma \\
&+  \int_{\Gamma_C}  \frac{c_{1\phi}|\beta(t)|^2 |\xi(t)| }{M(t)+K-u_+(t)} w_h^{2\kappa} w   \d \sigma
\end{align*}
for a.e. $t\in (0,T)$. In Part II, we showed that $w\geq 0$. Then by \eqref{eq:assumptionmu} ($\mu(x) >\mu_\ast>0$ for a.e. $x\in \Omega$) and  $\frac{1}{M(t)+K-u_+(t)}\leq \frac{1}{K}$,  we have
\begin{align*}
\int_\Omega \mu \frac{ \nabla u(t) \cdot \nabla (w_h^{2\kappa} w) }{M(t)+K-u_+(t)} \d x + \mu_\ast &\int_\Omega   \frac{w_h^{2\kappa} w |\nabla u_+(t)|^2 }{(M(t)+K-u_+(t))^2} \d  x \\
&\leq \int_\Omega \frac{|f_0(t)| w_h^{2\kappa} w }{K}  \d x + \int_{\Gamma_N} \frac{|f_N(t)|w_h^{2\kappa} w}{K} \d \sigma \\
&+  \int_{\Gamma_C}  \frac{c_{0\varphi} + c_{1\varphi}|\beta(t)| + c_{2\varphi}|\xi(t)|}{K} w_h^{2\kappa} w   \d \sigma\\
&+  \int_{\Gamma_C}  \frac{c_{1\phi}|\beta(t)|^2 |\xi(t)|}{K}   w_h^{2\kappa} w \d \sigma
\end{align*}
for a.e. $t\in (0,T)$. Since $\mu_\ast \int_\Omega   \frac{w_h^{2\kappa} w |\nabla u(t)|^2 }{(M(t)+K-u(t))^2} \d x \geq 0$ by $w,w_h\geq0$, we may write 
\begin{align*}
\int_\Omega \mu \frac{ \nabla u(t) \cdot \nabla (w_h^{2\kappa} w) }{M(t)+K-u_+(t)} \d x  &\leq \int_\Omega \frac{|f_0(t)|w_h^{2\kappa} w }{K}  \d x + \int_{\Gamma_N} \frac{|f_N(t)|w_h^{2\kappa} w}{K} \d \sigma \\
&+  \int_{\Gamma_C}  \frac{c_{0\varphi} + c_{1\varphi}|\beta(t)| +c_{2\varphi}|\xi(t)|+c_{1\phi}|\beta(t)|^2 |\xi(t)|}{K} w_h^{2\kappa} w   \d \sigma\notag
\end{align*}
for a.e. $t\in(0,T)$.  From the definition of $K$, we have that
\begin{align*}
\int_\Omega \mu \frac{ \nabla u(t) \cdot \nabla (w_h^{2\kappa} w) }{M(t)+K-u_+(t)} \d x  &\leq \mu_\ast \int_\Omega w_h^{2\kappa} w \d x + \mu_\ast \int_{\Gamma_N} w_h^{2\kappa} w \d \sigma +  \mu_\ast \int_{\Gamma_C}  w_h^{2\kappa} w  \d \sigma
\end{align*}
for a.e. $t\in (0,T)$.
We define the above inequality as
\begin{align*}
	A_1 \leq A_2 + A_3 + A_4 
\end{align*}
to estimate the terms separately and then combine them before using a bootstrap argument to obtain the desired bound. We start by looking at $A_1$. From
\begin{align*}
	\nabla (w_h^{2\kappa} w) &= w \nabla (w_h^{2\kappa}) + w_h^{2\kappa} \nabla w =  2\kappa w_h^{2\kappa-1} w\nabla w_h + w_h^{2\kappa} \nabla w,
\end{align*}
we have
\begin{align*}
A_1 &=  2\kappa \int_\Omega \mu  w_h^{2\kappa-1}w  \frac{ \nabla u(t) \cdot  \nabla w_h }{M(t)+K-u_+(t)}    \d x + \int_\Omega \mu w_h^{2\kappa} \frac{ \nabla u(t) \cdot  \nabla w }{M(t)+K-u_+(t)}   \d x  \\
&=: A_1^1 + A_1^2
\end{align*}
for a.e. $t\in (0,T)$. To rewrite $A_1^1$ in terms of $w$ and $w_h$, we first observe that $w_h$ can be written as
\begin{align*}
w_h  &= w\mathds{1}_{\{w< h\}} + h \mathds{1}_{\{w\geq h\}} = (w-h)\mathds{1}_{\{w-h< 0\}} + h = -(w-h)_- + h.
\end{align*}
Then according to Proposition \ref{prop:nabla}
\begin{align}\label{eq:grad_w_h}
\nabla w_h = -\nabla (w-h)_- = \mathds{1}_{\{ w<h\}} \nabla w.
\end{align}
As a consequence of the definition of $w$ \eqref{eq:def_w}, \eqref{eq:grad_w_h}. and Proposition \ref{prop:nabla}, we deduce
\begin{align*}
\frac{ \nabla u(t) \cdot  \nabla w}{M(t)+K-u_+(t)}  \mathds{1}_{\{ w<h\}} &= \frac{ \nabla u(t) \cdot  \nabla u_+(t)}{(M(t)+K-u_+(t))^2}  \mathds{1}_{\{ w<h\}}\\
&=\frac{ |\nabla u_+(t)|^2}{(M(t)+K-u_+(t))^2}  \mathds{1}_{\{ w<h\}} \\
&= |\nabla w|^2 \mathds{1}_{\{ w<h\}}\\
&= |\nabla w_h|^2
\end{align*}
for a.e. $t\in (0,T)$. Gathering these calculations, $A_1^1$ becomes
\begin{align*}
A_1^1 &= 2\kappa \int_\Omega \mu  w_h^{2\kappa-1}w  |\nabla w_h|^2     \d x.
\end{align*}
The term $A_1^2$ is treated the same way. Hence, utilizing the lower bound of $\mu$ \eqref{eq:assumptionmu}, \eqref{eq:grad_w_h}, and $w\geq w_h$, the term $A_1$ reads 
\begin{align*}
A_1 &\geq 2\kappa\mu_\ast \int_\Omega  w_h^{2\kappa} |\nabla w_h|^2 \d x  + \mu_\ast  \int_\Omega  w_h^{2\kappa}  |\nabla w|^2   \d x\\
&\geq \mu_\ast(2\kappa+1) \int_\Omega  w_h^{2\kappa}  |\nabla w_h|^2 \d x \\
&=\mu_\ast\frac{2\kappa+1}{(\kappa+1)^2} \int_\Omega |\nabla (w_h^{\kappa+1})|^2 \d x.
\end{align*}
Next, we take a closer look at $A_2$. Applying Young's inequality with $\frac{1}{2(\kappa+1)} + \frac{2\kappa+1}{2(\kappa+1)} = 1$ yields
\begin{align*}
A_2 &\leq \mu_\ast \int_\Omega w^{2\kappa+1}\d x \leq \frac{\mu_\ast}{2\epsilon_1(\kappa+1)}|\Omega|+ \mu_\ast \epsilon_1\frac{2\kappa+1}{2(\kappa+1)} \int_\Omega w^{2(\kappa+1)} \d x ,
\end{align*}
where $\epsilon_1 >0$ is chosen later. For $A_3$, we apply  Young's inequality  with $\frac{1}{2(\kappa+1)} + \frac{2\kappa+1}{2(\kappa+1)} = 1$ and Proposition \ref{prop:trace} to obtain
\begin{align*}
A_3 &\leq \frac{\mu_\ast}{2\epsilon_2(\kappa+1)}|\Gamma_N|+ \mu_\ast \epsilon_2\frac{2\kappa+1}{2(\kappa+1)} \int_{\Gamma_N} w^{2(\kappa+1)} \d x \\
&\leq \frac{\mu_\ast}{2\epsilon_2(\kappa+1)}|\Gamma_N|+ c_0 \mu_\ast \epsilon_2 \epsilon_3\frac{2\kappa+1}{2(\kappa+1)} \int_\Omega |\nabla (w^{k+1})|^2 \d x + \frac{\mu_\ast \epsilon_2 c_0}{\epsilon_3} \frac{2\kappa+1}{2(\kappa+1)}  \int_\Omega w^{2(\kappa+1)} \d x,
\end{align*}
where we will choose $\epsilon_2, \epsilon_3>0$ later.
We follow the same procedure for $A_4$ as for $A_3$. This gives us 
\begin{align*}
A_4 &\leq \frac{\mu_\ast}{2\epsilon_4(\kappa+1)} |\Gamma_C| + \frac{\mu_\ast\epsilon_4(2\kappa+1)}{2(\kappa+1)}  \Big( c_0\epsilon_5
\norm{\nabla(w^{\kappa+1})}_{L^2(\Omega)}^2   +  \frac{c_0}{\epsilon_5}
\norm{w^{\kappa+1}}_{L^2(\Omega)}^2\Big)
\end{align*}
for $\epsilon_4,\epsilon_5>0$ chosen later. We gather the calculations for $A_1$, $A_2$, $A_3$, and $A_4$ to obtain 
\begin{align*}
\mu_\ast\frac{2\kappa+1}{(\kappa+1)^2} &\int_\Omega |\nabla (w_h^{\kappa+1})|^2 \d x \\
&\leq \frac{\mu_\ast}{2\epsilon_1(\kappa+1)} |\Omega| + \frac{\mu_\ast\epsilon_1(2\kappa+1)}{2(\kappa+1)}  \norm{w^{\kappa+1}}_{L^2(\Omega)}^2 \\
&+\frac{\mu_\ast}{2\epsilon_2(\kappa+1)} |\Gamma_N| + \frac{\mu_\ast\epsilon_2(2\kappa+1)}{2(\kappa+1)}  \Big( c_0\epsilon_3
\norm{\nabla(w^{\kappa+1})}_{L^2(\Omega)}^2   +  \frac{c_0}{\epsilon_3}
\norm{w^{\kappa+1}}_{L^2(\Omega)}^2\Big)\\
&+ \frac{\mu_\ast}{2\epsilon_4(\kappa+1)} |\Gamma_C| + \frac{\mu_\ast\epsilon_4(2\kappa+1)}{2(\kappa+1)}  \Big( c_0\epsilon_5
\norm{\nabla(w^{\kappa+1})}_{L^2(\Omega)}^2   +   \frac{c_0}{\epsilon_5}
\norm{w^{\kappa+1}}_{L^2(\Omega)}^2\Big).
\end{align*}
Choosing $\epsilon_1 = \epsilon_2 = \epsilon_4 = \frac{1}{(\kappa+1)}$ and $\epsilon_3 =
\epsilon_5 =   \frac{2}{3c_0}$	implies
\begin{align*}
&\norm{\nabla(w_h^{\kappa+1})}_{L^2(\Omega)}^2   \leq \frac{2}{3}\norm{\nabla(w^{\kappa+1})}_{L^2(\Omega)}^2  +  c_0\frac{(\kappa+1)^2}{2\kappa+1}  |\Omega|+ c_0 \norm{w^{\kappa+1}}_{L^2(\Omega)}^2.
\end{align*}
We observe that
\begin{align*}
\nabla(w_h^{\kappa+1})(x) = \mathds{1}_{\{w< h\}}(x)
\nabla(w^{\kappa+1})(x) \rightarrow \nabla(w^{\kappa+1})(x) 
\end{align*}
as $h \rightarrow \infty$ for a.e. $x\in \Omega$ . An application of Fatou's lemma while noting that $\frac{(\kappa+1)^2}{2\kappa+1} \geq 1$ for any $\kappa>0$ yields
\begin{align}\label{eq:w_est_1}
&\norm{\nabla(w^{\kappa+1})}_{L^2(\Omega)}^2 \leq c_0\frac{(\kappa+1)^2}{2\kappa+1}  ( \norm{w^{\kappa+1}}_{L^2(\Omega)}^2  + |\Omega|).
\end{align}
Defining
\begin{align}\label{eq:def_q}
q^{\kappa+1}:= w^{\kappa+1}+ 1
\end{align}
implies $q^{2(\kappa+1)}= (w^{\kappa+1}+ 1)^2 =w^{2(\kappa+1)} + 2w^{\kappa+1}+ 1 \geq w^{2(\kappa+1)} + 1 $ since $w\geq 0$ and
\begin{align*}
\int_\Omega w^{2(\kappa+1)} \d x  + |\Omega| \leq \int_\Omega q^{2(\kappa+1)} \d x.
\end{align*}
Since $\nabla (q^{\kappa+1}) = \nabla (w^{\kappa+1})$, 
we add $\norm{q^{\kappa+1}}_{L^2(\Omega)}^2$ to both sides of 
the  inequality \eqref{eq:w_est_1} to obtain
\begin{align*}
\norm{q^{\kappa+1}}_{H^1(\Omega)}^2  &\leq c_1 \frac{(\kappa+1)^{2}}{2\kappa+1} \norm{q^{\kappa+1}}_{L^2(\Omega)}^2
\end{align*}
for a constant $c_1 >0$. From the above inequality, we have that $q^{\kappa+1} \in H^1_0(\Omega)$ provided that $q^{\kappa+1} \in L^2(\Omega)$. We will show this inductively. From the Sobolev embedding $H^1_0(\Omega) \subset L^{p^\ast}(\Omega)$ in Proposition \ref{prop:sobolev_embedding}, we have
\begin{align*}
\norm{q}_{L^{p^\ast (\kappa+1)}(\Omega)}&= \norm{q^{\kappa+1}}_{L^{p^\ast}(\Omega)}^{\frac{1}{\kappa+1}} \\  \notag
&\leq c_2^{\frac{1}{\kappa+1}} \norm{q^{\kappa+1}}_{H^1(\Omega)}^{\frac{1}{\kappa+1}}\\  \notag
&\leq c_3^{\frac{1}{\kappa+1}}  \Big( \frac{(\kappa+1)^2}{2\kappa+1} \Big)^{\frac{1}{2(\kappa+1)}} \norm{q}_{L^{2(\kappa+1)}(\Omega)}
\end{align*}
for $c_2>0$ and $c_3\geq 1$. We will define a sequence inductively to show the desired estimate. 
Starting with $\kappa := \kappa_0 = 0$ implies 
\begin{align}\label{eq:k_0}
\norm{q}_{L^{p^\ast }(\Omega)}
&\leq c_3   \norm{q}_{L^{2}(\Omega)}  ,
\end{align}
where we notice that $\Big( \frac{(\kappa_0+1)^2}{2\kappa_0+1} \Big)^{\frac{1}{2(\kappa_0+1)}} =1$. From the definition of $q$ \eqref{eq:def_q}, we have for $\kappa = \kappa_0 = 0$ that
\begin{align*}
    q = w + 1.
\end{align*}
We utilize Young's inequality to obtain
\begin{align*}
    q^2 = (w + 1)^2 \leq 2 (w^2 + 1),
\end{align*}
which implies
\begin{align*}
    \norm{q}_{L^{p^\ast }(\Omega)}
            &\leq c_3 \Big(   \int_\Omega q^2 \d x \Big)^{1/2} \leq \sqrt{2}c_3 \Big(\int_\Omega  w^2 \d x \Big)^{1/2} + \sqrt{2}c_3|\Omega|^{1/2} =\sqrt{2}c_3 (\norm{w}_{L^2(\Omega)} + |\Omega|^{1/2}).
\end{align*}
From Part I, we have that there exists a constant $c_4>0$ such that
\begin{align}\label{eq:step1}
\norm{q}_{L^{p^\ast }(\Omega)}
&\leq  c_3c_4.
\end{align}
We continue by taking $\kappa = \kappa_1>0$ satisfying $2(\kappa_1 +1) = p^\ast$ and using \eqref{eq:k_0}. This yields
\begin{align}\label{eq:step2}
\norm{q}_{L^{p^\ast (\kappa_1+1)}(\Omega)} &\leq c_3^{\frac{1}{\kappa_1+1}}  \Big( \frac{(\kappa_1+1)^2}{2\kappa_1+1} \Big)^{\frac{1}{2(\kappa_1+1)}}  \norm{q}_{L^{2(\kappa_1+1)}(\Omega)}  \\  \notag
&=c_3^{\frac{1}{\kappa_1+1}} \Big( \frac{(\kappa_1+1)^2}{2\kappa_1+1} \Big)^{\frac{1}{2(\kappa_1+1)}}  \norm{q}_{L^{p^\ast}(\Omega)} \\  \notag
&\leq  c_3^{\frac{1}{\kappa_1+1}} \Big( \frac{(\kappa_1+1)^2}{2\kappa_1+1} \Big)^{\frac{1}{2(\kappa_1+1)}}  c_3  c_4. 
\end{align}
We notice that $\Big( \frac{(\kappa_0+1)^2}{2\kappa_0+1} \Big)^{\frac{1}{2(\kappa_0+1)}} = 1$ and $\frac{1}{\kappa_0 + 1} = 1$, then 
\begin{align}\label{eq:step3}
\norm{q}_{L^{p^\ast (\kappa_1+1)}(\Omega)} 
&\leq  c_3^{\frac{1}{\kappa_0+1}}  c_3^{\frac{1}{\kappa_1+1}} \Big( \frac{(\kappa_0+1)^2}{2\kappa_0+1} \Big)^{\frac{1}{2(\kappa_0+1)}} \Big( \frac{(\kappa_1+1)^2}{2\kappa_1+1} \Big)^{\frac{1}{2(\kappa_1+1)}}   c_4\\ \notag
&= c_3^{\sum_{j=0}^1 \frac{1}{\kappa_j+1}} \prod_{j=0}^1\Big( \frac{(\kappa_j+1)^2}{2\kappa_j+1} \Big)^{\frac{1}{2(\kappa_j+1)}}   c_4.
\end{align}
For the next iteration, we let $\kappa_2>0$ be such that $2(\kappa_2+1) = p^\ast (\kappa_1 + 1)$ and use the estimate \eqref{eq:step3} to obtain
\begin{align}\label{eq:step4}
    \norm{q}_{L^{p^\ast (\kappa_2+1)}(\Omega)} &\leq  c_3^{\frac{1}{\kappa_2+1}}  \Big( \frac{(\kappa_2+1)^2}{2\kappa_2+1} \Big)^{\frac{1}{2(\kappa_2+1)}}  \norm{q}_{L^{2(\kappa_2+1)}(\Omega)} \\ \notag
    &= c_3^{\frac{1}{\kappa_2+1}}  \Big( \frac{(\kappa_2+1)^2}{2\kappa_2+1} \Big)^{\frac{1}{2(\kappa_2+1)}}  \norm{q}_{L^{p^\ast (\kappa_1+1)}(\Omega)}\\ \notag
    &\leq  c_3^{\frac{1}{\kappa_2+1}}  \Big( \frac{(\kappa_2+1)^2}{2\kappa_2+1} \Big)^{\frac{1}{2(\kappa_2+1)}} \Big( c_3^{\sum_{j=0}^1 \frac{1}{\kappa_j+1}} \prod_{j=0}^1\Big( \frac{(\kappa_j+1)^2}{2\kappa_j+1} \Big)^{\frac{1}{2(\kappa_j+1)}}   c_4\Big)\\ \notag
    &=c_3^{\sum_{j=0}^2 \frac{1}{\kappa_j+1}} \prod_{j=0}^2\Big( \frac{(\kappa_j+1)^2}{2\kappa_j+1} \Big)^{\frac{1}{2(\kappa_j+1)}}   c_4.
\end{align}
We continue by constructing the following sequence
\begin{align*}
\kappa_0 : \hspace{1.2cm}\kappa_0 &= 0\\
\kappa_1 : 2(\kappa_1 +1) &= p^\ast\\
\kappa_2  :  2 (\kappa_2 +1) &=  p^\ast(\kappa_1 +1)\\
\kappa_3  :  2 (\kappa_3 +1) &=  p^\ast(\kappa_2 +1)\\
&  \vdots
\end{align*}
We observe that the sequence is increasing. Indeed, for $p^\ast >2$ and any $\ell \geq 1$, we have  
\begin{align*}
(\kappa_\ell + 1) = \Big(\frac{p^\ast}{2}\Big)^\ell.
\end{align*}
Repeating the arguments in \eqref{eq:step1}-\eqref{eq:step4}, we obtain for any $\ell \geq 0$
\begin{align}\label{eq:passlimit_term_q}
\norm{w}_{L^{p^\ast (\kappa_\ell+1)}(\Omega)} \leq  c_3^{\sum_{j=0}^\ell \frac{1}{\kappa_j+1}}  \prod_{j=0}^\ell \Big( \frac{(\kappa_j+1)^2}{2\kappa_j+1} \Big)^{\frac{1}{2(\kappa_j+1)}}   c_4.
\end{align}
 We wish to pass the limit $\ell \rightarrow \infty$. To this extent, we have to check the convergence of series and product on the right-hand side. Considering first $\prod_{j=0}^\ell \Big( \frac{(\kappa_j+1)^2}{2\kappa_j+1} \Big)^{\frac{1}{2(\kappa_j+1)}}$, we have from the definition of the logarithm and
\begin{align*}
\bigg(\frac{(\kappa_j+1)^2}{2\kappa_j+1}\bigg)^{\frac{1}{2(\kappa_j+1)}} \geq 1 \ \ \  \text{  for any }  \ \ \  \kappa_j>0,
\end{align*}
that
\begin{align*}
0<\log \Big( \prod_{j=0}^\ell \Big( \frac{(\kappa_j+1)^2}{2\kappa_j+1} \Big)^{\frac{1}{2(\kappa_j+1)}}\Big) &= \frac{1}{2} \sum_{j=0}^\ell \frac{1}{\kappa_j+1} \log \Big( \frac{(\kappa_j+1)^2}{2\kappa_j+1} \Big)\\
&\leq\sum_{j=0}^\ell \frac{1}{\kappa_j+1} \log ( \kappa_j+1)\\
&=\log \Big( \frac{p^\ast}{2}\Big) \sum_{j=0}^\ell j\frac{ 2^j}{(p^\ast)^j}.
\end{align*}
We test the root criterion for convergence, which yields
\begin{align*}
\liminf_{j\rightarrow \infty} \Big(j\frac{ 2^j}{(p^\ast)^j} \Big)^\frac{1}{j} \leq \frac{2}{p^\ast }< 1.
\end{align*}
Thus, the term $\prod_{j=0}^\ell \Big( \frac{(\kappa_j+1)^2}{2\kappa_j+1} \Big)^{\frac{1}{2(\kappa_j+1)}}$ converges as $\ell \rightarrow \infty$.
Next, we have that the series in the term $c_3^{\sum_{j=0}^\ell \frac{1}{\kappa_j+1}}$ is a geometric series since $\frac{1}{\kappa_j+1} = \Big( \frac{2}{p^\ast}\Big)^j$ for any $j\geq 0$ with $\frac{2}{p^\ast}<1$. Indeed,
\begin{align*}
\lim_{\ell \rightarrow\infty} \sum_{j=0}^\ell{\frac{1}{\kappa_j+1}} &= \frac{1}{1-\frac{2}{p^\ast}}= \frac{p^\ast}{p^\ast-2}.
\end{align*}
We are now in a position to pass the limit in \eqref{eq:passlimit_term_q}. 
Defining $r_\ell := p^\ast (\kappa_\ell+1)$, we have a sequence $\{r_\ell\}_\ell\subset (1,\infty)$ where $r_\ell \rightarrow \infty $ as $\ell \rightarrow \infty$. Thus, we may apply Lemma \ref{lemma:passingthelimit_bound} to obtain
\begin{align}\label{eq:estimate_on_q}
\norm{q}_{L^\infty(\Omega)}
&\leq c_5
\end{align}
for a constant $c_5>0$. We use \eqref{eq:def_q} to find the desired estimate. Indeed, for any $\kappa >0$
\begin{align*}
    q^{\kappa+1} = w^{\kappa+1}  + 1\geq w^{\kappa+1} \implies q\geq w.
\end{align*}
Consequently, from \eqref{eq:estimate_on_q}, we have
\begin{align}\label{eq:estimate_on_w}
\norm{w}_{L^\infty(\Omega)}
&\leq c_5,
\end{align}
as desired.
\subsection*{Part III.2: $\norm{w(t)}_{L^\infty(\Gamma_C)} +  \norm{w(t)}_{L^\infty(\Gamma_N)}\leq c_0$} The estimate follows directly by Part I-II and III.1 combined with Proposition \ref{prop:traceinfty}.
\subsection*{Part IV: $\norm{u_+}_{L^\infty_TL^\infty(\Bar{\Omega})} \leq c_0K$} 
In this part, we will combine the results from the three previous steps to obtain the desired estimate. We denote 
\begin{align*}
   h(x,t) :=    \frac{M(t) + K }{M(t) + K - u_+(x,t)}
\end{align*}
so that 
\begin{align*}
    w(x,t) = \log (h(x,t)).
\end{align*}
From the monotonicity and continuity of the logarithm, we have
\begin{align*}
    \sup_{x\in \Bar{\Omega}} \log (h(x,t)) =  \log ( \sup_{x\in \Bar{\Omega}} h(x,t))
\end{align*}
for a.e. $t\in (0,T)$. Consequently,
\begin{align*}
  \norm{w(t)}_{L^\infty(\Bar{\Omega})} 
   &= \log \Big(  \frac{M(t) + K }{ K} \Big)
\end{align*}
for a.e. $t\in (0,T)$. Therefore, it follows from \eqref{eq:estimate_on_w} in Part III.1-III.2 that 
\begin{align*}
     \frac{M(t) + K }{ K} \leq c_0
\end{align*}
for a.e. $t\in (0,T)$. Since $M(t) \leq M(t) + K$, we may conclude this proof.
\subsection*{Case 2: Estimating  $\norm{u_-}_{L^\infty_TL^\infty(\Bar{\Omega})} \leq c_0K$} Since $u_- \geq 0$,  we can utilize the same proof as for Case 1 by interchanging $u_+$ with $u_-$.
\end{proof}
\subsection*{Step 5 \textit{(Scheme for the approximated solution to \eqref{eq:eq1})}.}
For $n\in \N$, let $(u^{n-1},\beta^{n-1}) \in L^\infty_T (H^1_0(\Omega) \cap L^\infty(\Omega)) \times C([0,T];L^2(\Gamma_C)) \cap  L^\infty_T L^\infty(\Gamma_C)$ be known. We construct approximated solutions $\{(u^n,\beta^n)\}_{n\geq 1} \subset L^\infty_T (H^1_0(\Omega) \cap L^\infty(\Omega)) \times C([0,T];L^2(\Gamma_C)) \cap  L^\infty_T L^\infty(\Gamma_C)$ to \eqref{eq:eq1}, where $(u^n,\beta^n)$ is a solution to the scheme
\begin{subequations}\label{eq:u_beta_n_eq}
\begin{align}\label{eq:u_n}
\int_\Omega &\mu \nabla u^n(t) \cdot \nabla (v-u^n(t)) \d x\\ \notag
&+ \int_{\Gamma_C} [\varphi(\beta^{n-1}(t),  u^{n-1}(t),  v)-\varphi(\beta^{n-1}(t), u^{n-1}(t), u^n(t))] \d \sigma\\ \notag
&+ \int_{\Gamma_C} [\phi(\beta^{n-1}(t),  u^{n-1}(t),  v)-\phi(\beta^{n-1}(t), u^{n-1}(t), u^n(t))] \d \sigma\\
&\geq \int_{\Omega} f_0  (v-u^n(t)) \d x+ \int_{\Gamma_N} f_N ( v- u^n(t)) \d \sigma, \notag\\
\beta^n(t) &= \beta_0 + \int_0^t  H(\beta^n(s), u^n(s)) \d s
\label{eq:beta_eq}
\end{align}
\end{subequations}
for all $v\in H^1_0(\Omega)$, a.e. $t\in(0,T)$ for the initial guess $\beta^0 := \beta_0 \in L^\infty(\Gamma_C)$ and for any $u^0 \in L^\infty_T (H^1_0(\Omega) \cap L^\infty(\Omega))$.

\subsubsection*{Step 5.1 \textit{(Existence and uniqueness of $(u^n,\beta^n)\in  L^\infty_T (H^1_0(\Omega) \cap L^\infty(\Omega)) \times C([0,T];L^2(\Gamma_C)) \cap L^\infty_TL^\infty(\Gamma_C)$ to \eqref{eq:u_n}-\eqref{eq:beta_eq} for all $n\in \N$)}.} 
We establish existence and uniqueness by induction on $n$.  From Proposition \ref{prop:estn} and \ref{prop:u_bound} with $u_{\xi\beta} = u^n$, $\xi = u^{n-1}$, and $\beta = \beta^{n-1}$,  we obtain
\begin{align}\label{eq:est_u_n_V}
\norm{u^n}_{L^\infty_T H^1(\Omega)}     &\leq c(1 +   \norm{f_0}_{L^\infty_T L^\infty(\Omega)} + \norm{f_N}_{L^\infty_T L^\infty(\Gamma_N)}) \\ \notag
&+\frac{c_{1\varphi} c_0}{\mu_\ast} \norm{\beta^{n-1}}_{L^\infty_TL^\infty(\Gamma_C)}+ \frac{c_{2\varphi} c_0}{\mu_\ast} \norm{u^{n-1}}_{L^\infty_T  H^1(\Omega)} \\ \notag
&+ \frac{c_{1\phi} c_0}{\mu_\ast} \norm{\beta^{n-1}}_{L^\infty_TL^\infty(\Gamma_C)}^2 \norm{u^{n-1}}_{L^\infty_T  H^1(\Omega)},\\
\label{eq:estwinfn1}
\norm{u^n}_{L^\infty_T L^\infty(\Omega)} &\leq c (1+\norm{f_0}_{L^\infty_TL^\infty(\Omega)} +\norm{f_N}_{L^\infty_TL^\infty(\Gamma_N)})\\ \notag
&+ \frac{c_{1\varphi}c_0}{\mu_\ast}\norm{\beta^{n-1}}_{L^\infty_TL^\infty(\Gamma_C)} +\frac{c_{2\varphi}c_0}{\mu_\ast}\norm{u^{n-1}}_{L^\infty_TL^\infty(\Omega)}\\
& +\frac{c_{1\phi}c_0}{\mu_\ast}\norm{\beta^{n-1}}_{L^\infty_TL^\infty(\Gamma_C)}^2 \norm{u^{n-1}}_{L^\infty_TL^\infty(\Omega)}\notag
\end{align}
for all $n\in \N$. Moreover, applying Lemma \ref{lemma:calculation} to \eqref{eq:beta_eq} yields
\begin{align*}
\norm{\beta^n(t)}_{L^\infty(\Gamma_C)} &\leq \norm{\beta_0}_{L^\infty(\Gamma_C)} + c_{0\beta}T+   c_{3\beta}c_0  \norm{u^n}_{L^\infty_TL^\infty(\Omega)}^2  \int_0^t  \norm{\beta^n(s)}_{L^\infty(\Gamma_C)} \d s
\end{align*}
for all $n\in\N$ and a.e. $t\in (0,T)$. A standard Gr\"{o}nwall argument (see, e.g., \cite[Appendix B.2]{Evans2010}) implies
\begin{align}\label{eq:est_beta_useful}
\norm{\beta^n}_{L^\infty_TL^\infty(\Gamma_C)} &\leq \norm{\beta_0}_{L^\infty(\Gamma_C)}  (1+ cT \norm{u^n}^2_{L^\infty_TL^\infty(\Omega)} \mathrm{e}^{cT \norm{u^n}^2_{L^\infty_TL^\infty(\Omega)}} ) \\
&+ c_{0\beta}T(1+ cT \norm{u^n}^2_{L^\infty_TL^\infty(\Omega)} \mathrm{e}^{cT \norm{u^n}^2_{L^\infty_TL^\infty(\Omega)}}) \notag
\end{align}
for all $n\in \N$. 
\\
\\
\indent \textbf{Part I ($n=1$).} For $n=1$ with the initial guess $\beta^0 := \beta_0 \in L^\infty(\Gamma_C)$ and for any $u^0 \in L^\infty_T ( H^1_0(\Omega) \cap L^\infty(\Omega))$, we have from Step 2 that $u^1 \in L^\infty_T  H^1_0(\Omega)$ is the unique solution to \eqref{eq:u_n} for $n=1$. 
In addition, we have the following uniform estimate.
\begin{prop}\label{prop:n1}
Assume that \hyperref[assumptionvarphi]{$(H1)$}-\hyperref[assumptiophi]{$(H2)$} and \eqref{eq:assumptions_mu}-\eqref{eq:assumptionondata} hold. Let $\beta^0:=\beta_0 \in L^\infty(\Gamma_C)$, $u^0\in L^\infty_T ( H^1_0(\Omega) \cap L^\infty(\Omega))$, and $u^1$ be the solution to \eqref{eq:u_n} for $n=1$. Then 
\begin{align*}
\norm{u^1}_{L^\infty_T  H^1(\Omega)} +
\norm{u^1}_{L^\infty_T L^\infty(\Omega)}  &\leq c,
\end{align*}
where $c=c(\norm{(f_0,f_N,\beta_0,u^0)}_{L^\infty_T L^\infty(\Omega) \times L^\infty_T L^\infty(\Gamma_N) \times L^\infty(\Gamma_C)\times L^\infty_T ( H^1_0(\Omega) \cap L^\infty(\Omega))})$.
\end{prop}
\begin{proof}
From \eqref{eq:est_u_n_V} and \eqref{eq:estwinfn1} for $n=1$, we have
\begin{align*}
\norm{u^1}_{L^\infty_T H^1(\Omega)}     &\leq c(1 +   \norm{f_0}_{L^\infty_T L^\infty(\Omega)} + \norm{f_N}_{L^\infty_T L^\infty(\Gamma_N)}) +\frac{c_{1\varphi} c_0}{\mu_\ast} \norm{\beta_0}_{L^\infty(\Gamma_C)}\\
&+ \frac{c_{2\varphi} c_0}{\mu_\ast} \norm{u^{0}}_{L^\infty_T H^1(\Omega)} + \frac{c_{1\phi} c_0\norm{\beta_0}_{L^\infty(\Gamma_C)}^2 }{\mu_\ast} \norm{u^0}_{L^\infty_T  H^1(\Omega)},
\\
\norm{u^1}_{L^\infty_T L^\infty(\Omega)} &\leq c (1+\norm{f_0}_{L^\infty_TL^\infty(\Omega)} +\norm{f_N}_{L^\infty_TL^\infty(\Gamma_N)}) + \frac{c_{1\varphi}c_0}{\mu_\ast}\norm{\beta_0}_{L^\infty(\Gamma_C)} \\
&+\frac{c_{2\varphi}c_0}{\mu_\ast}\norm{u^{0}}_{L^\infty_TL^\infty(\Omega)} +\frac{c_{1\phi}c_0\norm{\beta_0}_{L^\infty(\Gamma_C)}^2 }{\mu_\ast}\norm{u^{0}}_{L^\infty_TL^\infty(\Omega)}.
\end{align*}
We utilize the smallness-condition \eqref{eq:smallness} to obtain the desired estimate.
\end{proof}
We next define the complete metric space
\begin{align}\label{eq:constraction_space}
X_T(a) := \{ v \in C([0,T];L^2(\Gamma_C)) : \text{there is a } a\in \R_+ \text{ such that } \norm{v}_{L^\infty_TL^2(\Gamma_C)}\leq  a \}
\end{align}
and the application $\Lambda :X_T(a)  \rightarrow X_T(a)$ 
\begin{align}\label{eq:Lambda1}
\Lambda \beta^1(t) = \beta_0 + \int_0^t  H(\beta^1(s), u^1(s)) \d s.
\end{align}
We verify that $\beta^1$ is a solution to \eqref{eq:beta_eq} for $n=1$ in the next lemma. 
\begin{lem}\label{lemma:lambda_fixedpoint}
Assume that \hyperref[assumptionvarphi]{$(H1)$}-\hyperref[assumptionH]{$(H3)$} and \eqref{eq:assumptions_mu}-\eqref{eq:assumptionondata} hold. Let $u^1 \in L^\infty_T ( H^1_0(\Omega) \cap L^\infty(\Omega))$ be a solution to \eqref{eq:u_n} for $n=1$ with $\beta^0=\beta_0 \in L^\infty(\Gamma_C)$ and $u^0 \in L^\infty_T ( H^1_0(\Omega) \cap L^\infty(\Omega))$. Then the operator   $\Lambda :X_T(a)  \rightarrow X_T(a)$ defined by \eqref{eq:Lambda1} has a unique fixed point.    
\end{lem}
\begin{proof}
The proof relies on the Banach fixed-point theorem. We therefore need to verify that the map is indeed well-defined and that it is a contractive mapping on $X_T(a)$. For the sake of presentation, we split the proof into two steps.
\\
\indent \textbf{Step i}  \textit{(The operator $\Lambda$ is well-defined on $X_T(a)$)}. 	Indeed, $\Lambda \beta^1 \in X_T(a)$. We first prove that for given $\beta^1 \in X_T(a)$, then $\norm{\Lambda \beta^1}_{L^\infty_T L^2(\Gamma_C)} \leq a$. 
Using Lemma \ref{lemma:calculation}, H\"{o}lder's inequality, and Proposition \ref{prop:traceinfty} together with Proposition \ref{prop:n1} implies
\begin{align}\label{eq:lamda_line1}
    \norm{\Lambda \beta^1(t)}_{L^2(\Gamma_C)}
    &\leq \norm{\beta_0}_{L^2(\Gamma_C)} +\int_0^t ( c_{0\beta}  +c_{3\beta} \norm{u^1(s)}_{L^\infty(\Gamma_C)}^2 \norm{\beta^1(s)}_{L^2(\Gamma_C)})\d s  \\  \notag
    &\leq \norm{\beta_0}_{L^2(\Gamma_C)} + Tc_{0\beta} + T c_{3\beta} c_0\norm{u^1}_{L^\infty_TL^\infty(\Omega)}^2\norm{\beta^1}_{L^\infty_TL^2(\Gamma_C)}.
\end{align}
We choose $T>0$ such that
\begin{align}\label{eq:T}
T \sim  \frac{1}{c(\norm{(f_0,f_N,\beta_0,u^0)}^2_{L^\infty_T L^\infty(\Omega) \times L^\infty_T L^\infty(\Gamma_N) \times L^\infty(\Gamma_C)\times L^\infty_T ( H^1(\Omega) \cap L^\infty(\Omega))})}.
\end{align}
We may now choose $a\geq0$ so that it provides the desired upper bound.
\\
\indent
It remains to show that $\Lambda \beta^1(t)$ is continuous on $L^2(\Gamma_C)$ for given $\beta^1 \in  X_T(a)$ and $u^1 \in L^\infty_T ( H^1_0(\Omega) \cap L^\infty(\Omega))$. Let $t, t'\in [0,T]$, then with no loss of generality, we assume that $t < t'$. From Lemma \ref{lemma:calculation} and Proposition \ref{prop:traceinfty}, we deduce
\begin{align*}
\norm{\Lambda \beta^1(t') - \Lambda \beta^1(t)}_{L^2(\Gamma_C)} & \leq \int_t^{t'} | H(\beta^1(s),u^1(s))| \d s\\ 
&\leq   c_{0\beta}|t'-t|+ c_{3\beta}c_0 \norm{u^1}_{L^\infty_TL^\infty(\Omega)}^2 \norm{\beta^1}_{L^\infty_T L^2(\Gamma_C)}|t'-t|. 
\end{align*}
Now, Proposition \ref{prop:n1} implies
\begin{align*}
\norm{\Lambda \beta^1(t') - \Lambda \beta^1(t)}_{L^2(\Gamma_C)} 
&\leq c|t'-t|  .
\end{align*}
Passing the limit $|t'-t| \rightarrow 0$, we have that indeed $\Lambda \beta^1\in X_T(a)$.
\\
\indent\textbf{Step ii} \textit{(The application $\Lambda : X_T(a) \rightarrow X_T(a)$ is a contractive mapping)}. 	Let $\beta^1_j \in X_T(a)$, $j=1,2$, and let $ u^1 \in L^\infty_T ( H^1_0(\Omega) \cap  L^\infty(\Omega))$ be the unique solution to \eqref{eq:u_n} for $n=1$. We begin by introducing a new norm for the space $C([0,T];L^2(\Gamma_C))$
\begin{equation*}
\norm{z}_\omega = \max_{s\in[0,T]} \mathrm{e}^{-\omega s} \norm{z(s)}_{L^2(\Gamma_C)}
\end{equation*}
where $\omega>0$ is chosen later. We notice that $C([0,T];L^2(\Gamma_C))$ with $\norm{\cdot}_\omega$ is complete and $\norm{\cdot}_\omega$  is  equivalent to the norm on $C([0,T];L^2(\Gamma_C))$ (see, e.g., \cite[Section 4.2]{Cheney2001}).  We first apply \hyperref[assumptionH]{$(H3)$}$(ii)$, H\"{o}lder's inequality, and Proposition \ref{prop:traceinfty} to \eqref{eq:Lambda1}. This yields
\begin{align*}
&\norm{\Lambda \beta^1_1(t) - \Lambda \beta^1_2(t)}_{L^2(\Gamma_C)} \leq c_{3\beta} c_0\norm{u^1}_{L^\infty_TL^\infty(\Omega)}^2  \int_0^t \norm{\beta^1_1(s) -  \beta^1_2(s)}_{L^2(\Gamma_C)} \d s
\end{align*}
for a.e. $t\in (0,T)$. Then Proposition \ref{prop:n1} implies
\begin{align*}
\mathrm{e}^{-\omega t} \norm{\Lambda \beta^1_1(t) - \Lambda \beta^1_2(t)}_{L^2(\Gamma_C)} &\leq  c \int_0^t \mathrm{e}^{-\omega t} \mathrm{e}^{\omega s} \mathrm{e}^{-\omega s}  \norm{\beta^1_1(s) - \beta^1_2(s)}_{L^2(\Gamma_C)} ds\\ 
&\leq  c   \norm{\beta^1_1 - \beta^1_2}_\omega \int_0^t \mathrm{e}^{-\omega t} \mathrm{e}^{\omega s}  ds \\
&\leq \frac{c}{\omega}  \norm{\beta^1_1 - \beta^1_2}_\omega 
\end{align*}
for a.e. $t\in (0,T)$. Choosing $\omega > c$ implies that $\Lambda$ is a contraction on $X_T(a)$, and thus we may conclude by the Banach fixed-point theorem that $\beta^1 \in X_T(a)$ is the unique fixed point to \eqref{eq:Lambda1}. 
\end{proof}
In the proposition below, we find a boundedness result for $\beta^1$.
\begin{prop}\label{prop:est_beta}
Assume that \hyperref[assumptionvarphi]{$(H1)$}-\hyperref[assumptionH]{$(H3)$} and \eqref{eq:assumptions_mu}-\eqref{eq:assumptionondata} hold. Let $(u^1,\beta^1) \in  L^\infty_T ( H^1_0(\Omega) \cap L^\infty(\Omega)) \times C([0,T];L^2(\Gamma_C))$ be the solution to \eqref{eq:u_n}-\eqref{eq:beta_eq}, then
\begin{align*}
\norm{\beta^1}_{L^\infty_TL^\infty(\Gamma_C)} \leq c
\end{align*}
for $c=c(\norm{(f_0,f_N,\beta_0,u^0)}_{L^\infty_T L^\infty(\Omega) \times L^\infty_T L^\infty(\Gamma_N) \times L^\infty(\Gamma_C)\times L^\infty_T ( H^1(\Omega) \cap L^\infty(\Omega))})$.
\end{prop}
\begin{proof}
Taking $n=1$ in \eqref{eq:est_beta_useful} yields
\begin{align*}
\norm{\beta^1}_{L^\infty_TL^\infty(\Gamma_C)} &\leq \norm{\beta_0}_{L^\infty(\Gamma_C)}  (1+ cT \norm{u^1}^2_{L^\infty_TL^\infty(\Omega)} \mathrm{e}^{cT \norm{u^1}^2_{L^\infty_TL^\infty(\Omega)}})\\
&+c_{0\beta} T(1+ cT \norm{u^1}^2_{L^\infty_TL^\infty(\Omega)} \mathrm{e}^{cT \norm{u^1}^2_{L^\infty_TL^\infty(\Omega)}}).
\end{align*}
Keeping in mind Proposition \ref{prop:n1}, we choose $T>0$ as in \eqref{eq:T} to complete the proof.
\end{proof}

\textbf{Part II (Induction step).}
For the induction step, we let $(u^{n-1},\beta^{n-1}) \in L^\infty_T ( H^1_0(\Omega) \cap L^\infty(\Omega)) \times C([0,T];L^2(\Gamma_C)) \cap L^\infty_T L^\infty(\Gamma_C)$ be the solution to \eqref{eq:u_n}-\eqref{eq:beta_eq} for $n\in \N$. It then follows by Theorem \ref{thm:pre_existence_uniqueness} that $\{u^n\}_{n\geq 1}\subset L^\infty_T  H^1_0(\Omega)$ is the approximated solution to \eqref{eq:u_n}. We then find the following uniform estimate on $u^n$ for $n\in \N$.
\begin{prop}\label{prop:n}
Assume that \hyperref[assumptionvarphi]{$(H1)$}-\hyperref[assumptionH]{$(H3)$} and \eqref{eq:assumptions_mu}-\eqref{eq:assumptionondata} hold. Let $(u^{n-1},\beta^{n-1}) \in L^\infty_T ( H^1_0(\Omega) \cap L^\infty(\Omega)) \times C([0,T];L^2(\Gamma_C)) \cap L^\infty_T L^\infty(\Gamma_C)$ be the solution to \eqref{eq:u_n}-\eqref{eq:beta_eq} for any $n\in \N$ with $\beta^0 =\beta_0 \in L^\infty(\Gamma_C)$ and $u^0 \in L^\infty_T ( H^1_0(\Omega) \cap L^\infty(\Omega))$. Then, the following estimate holds
\begin{align}\label{eq:estVn1}
\norm{u^n}_{L^\infty_T  H^1(\Omega)} + \norm{u^n}_{L^\infty_TL^\infty(\Omega)}& \leq c
\end{align}
for any $n \in \N$ and with $c=c(\norm{(f_0,f_N,\beta_0,u^0)}_{L^\infty_T L^\infty(\Omega) \times L^\infty_T L^\infty(\Gamma_N)\times L^\infty(\Gamma_C)\times  L^\infty_T ( H^1(\Omega) \cap L^\infty(\Omega))})$.
\end{prop}
\begin{proof}
We begin by looking at \eqref{eq:est_beta_useful} at level $n-1$, that is
\begin{align*}
\norm{\beta^{n-1}}_{L^\infty_TL^\infty(\Gamma_C)} &\leq \norm{\beta_0}_{L^\infty(\Gamma_C)}  (1+ cT \norm{u^{n-1}}^2_{L^\infty_TL^\infty(\Omega)}\mathrm{e}^{cT \norm{u^{n-1}}^2_{L^\infty_TL^\infty(\Omega)}}) \\ \notag
&+ cT(1+ cT \norm{u^{n-1}}^2_{L^\infty_TL^\infty(\Omega)} \mathrm{e}^{cT \norm{u^{n-1}}^2_{L^\infty_TL^\infty(\Omega)}}) \\ \notag
&=  \norm{\beta_0}_{L^\infty(\Gamma_C)}+ cT\norm{\beta_0}_{L^\infty(\Gamma_C)} \norm{u^{n-1}}^2_{L^\infty_TL^\infty(\Omega)}\mathrm{e}^{cT \norm{u^{n-1}}^2_{L^\infty_TL^\infty(\Omega)}}  \\ \notag
&+ cT(1+ cT \norm{u^{n-1}}^2_{L^\infty_TL^\infty(\Omega)} \mathrm{e}^{cT \norm{u^{n-1}}^2_{L^\infty_TL^\infty(\Omega)}}).
\end{align*}
We define 
\begin{align}\label{eq:R}
R(n-1) := c (1&+\norm{\beta_0}_{L^\infty(\Gamma_C)} \norm{u^{n-1}}^2_{L^\infty_TL^\infty(\Omega)}\mathrm{e}^{cT \norm{u^{n-1}}^2_{L^\infty_TL^\infty(\Omega)}} \\
&+   \norm{u^{n-1}}^2_{L^\infty_TL^\infty(\Omega)} \mathrm{e}^{cT \norm{u^{n-1}}^2_{L^\infty_TL^\infty(\Omega)}}), \notag
\end{align}
for $n\geq 1$, so that
\begin{align}\label{eq:beta_est_not_squared}
\norm{\beta^{n-1}}_{L^\infty_TL^\infty(\Gamma_C)} &\leq  \norm{\beta_0}_{L^\infty(\Gamma_C)}+ T^\alpha R(n-1)
\end{align}
with $\alpha \geq 1$. Inserting \eqref{eq:beta_est_not_squared} into \eqref{eq:estwinfn1} yields
\begin{align*}
\norm{u^n}_{L^\infty_TL^\infty(\Omega)}&\leq c(1 +   \norm{f_0}_{L^\infty_TL^\infty(\Omega)} + \norm{f_N}_{L^\infty_TL^\infty(\Gamma_N)}) +\frac{c_{2\varphi} c_0}{\mu_\ast} \norm{u^{n-1}}_{L^\infty_TL^\infty(\Omega)} \\ 
&+\frac{c_{1\varphi} c_0}{\mu_\ast} (\norm{\beta_0}_{L^\infty(\Gamma_C)}  + T^\alpha R(n-1)) \\
&+\frac{c_{1\phi}c_0}{\mu_\ast}(\norm{\beta_0}_{L^\infty(\Gamma_C)}^2 + 2\norm{\beta_0}_{L^\infty(\Gamma_C)}T^\alpha R(n-1) +T^{2\alpha} R(n-1)^2) \norm{u^{n-1}}_{L^\infty_TL^\infty(\Omega)},
\end{align*}
where $\alpha \geq 1$. We observe that
\begin{align*}
R(n-1)^2 &= c(1 + \norm{\beta_0}_{L^\infty(\Gamma_C)}^2  \norm{u^{n-1}}^4_{L^\infty_TL^\infty(\Omega)}\mathrm{e}^{2cT \norm{u^{n-1}}^2_{L^\infty_TL^\infty(\Omega)}} \\
&+   \norm{u^{n-1}}^4_{L^\infty_TL^\infty(\Omega)} \mathrm{e}^{2cT \norm{u^{n-1}}^2_{L^\infty_TL^\infty(\Omega)}} \\
&+ 2 \norm{\beta_0}_{L^\infty(\Gamma_C)}  \norm{u^{n-1}}^2_{L^\infty_TL^\infty(\Omega)}\mathrm{e}^{cT \norm{u^{n-1}}^2_{L^\infty_TL^\infty(\Omega)}}\\
&+2\norm{u^{n-1}}^2_{L^\infty_TL^\infty(\Omega)} \mathrm{e}^{cT \norm{u^{n-1}}^2_{L^\infty_TL^\infty(\Omega)}}\\
&+2\norm{\beta_0}_{L^\infty(\Gamma_C)} \norm{u^{n-1}}^4_{L^\infty_TL^\infty(\Omega)}\mathrm{e}^{2cT \norm{u^{n-1}}^2_{L^\infty_TL^\infty(\Omega)}}).
\end{align*}
We simplify our notation by letting $\hat{c}=\hat{c}(\norm{(f_0,f_N,\beta_0,u^0)}_{L^\infty_T L^\infty(\Omega)\times L^\infty_T L^\infty(\Gamma_N)\times L^\infty(\Gamma_C)\times L^\infty_TL^\infty(\Omega)}) >0$ and
\begin{align*}
C_\ell(z) &:= \hat{c}\sum_{k=1}^\ell  z^k
\end{align*}
for $\ell\geq 1$. For each $n\in \N$, we define
\begin{align*}
\gamma_0 &:= T C_5(\norm{u^{0}}_{ L^\infty_T L^\infty(\Omega)})\\
\gamma_1 &:= T C_5(\norm{u^{1}}_{L^\infty_T L^\infty(\Omega)})\\
&\vdots\\
\gamma_{n-1} &:= T C_5(\norm{u^{n-1}}_{L^\infty_T L^\infty(\Omega)})
\end{align*}
for some $T>0$ small enough such that $\gamma_{n-1} < 1$, where $u^n$ will be defined inductively. We will now verify that we have a convergent sequence where $\gamma_{n-1}$ does not shrink to zero. Indeed, gathering the estimates and simplified notation above, the inequality is reduced to
\begin{align*}
\norm{u^n}_{L^\infty_TL^\infty(\Omega)} &\leq c(1 +   \norm{f_0}_{L^\infty_TL^\infty(\Omega)} + \norm{f_N}_{L^\infty_TL^\infty(\Gamma_N)} +\norm{\beta_0}_{L^\infty(\Gamma_C)}) \\
&+ \frac{c_{2\varphi} c_0+c_{1\phi}c_0\norm{\beta_0}_{L^\infty(\Gamma_C)}^2}{\mu_\ast}\norm{u^{n-1}}_{L^\infty_TL^\infty(\Omega)} + \gamma_{n-1}\\
&= \hat{c} + \frac{c_{2\varphi} c_0+c_{1\phi}c_0 \norm{\beta_0}_{L^\infty(\Gamma_C)}^2}{\mu_\ast}\norm{u^{n-1}}_{L^\infty_TL^\infty(\Omega)} + \gamma_{n-1}.
\end{align*}
We will soon iterate over $n\in\N$, but first, we will simplify our notation a bit more
\begin{align*}
\delta &:=\frac{c_{2\varphi} c_0+c_{1\phi}c_0\norm{\beta_0}_{L^\infty(\Gamma_C)}^2}{\mu_\ast},
\end{align*}
where $\delta <1$ as a consequence of the smallness-condition \eqref{eq:smallness}.
Consequently, we obtain
\begin{align}\label{eq:beforeiteration}
\norm{u^n}_{L^\infty_TL^\infty(\Omega)} &\leq \hat{c} +\delta\norm{u^{n-1}}_{L^\infty_TL^\infty(\Omega)} + \gamma_{n-1}.
\end{align}
Iterating over $n\in\N$, we observe that $\gamma_{n-1}$ only depends on the initial data, i.e.,
\begin{align*}
\norm{u^1}_{L^\infty_TL^\infty(\Omega)} &\leq \hat{c} +\delta\norm{u^{0}}_{L^\infty_TL^\infty(\Omega)} + \gamma_{0},\\
\norm{u^2}_{L^\infty_TL^\infty(\Omega)} &\leq \hat{c} +\delta\norm{u^{1}}_{L^\infty_TL^\infty(\Omega)} + \gamma_{1}\leq \hat{c}(1+\delta) +\delta^2\norm{u^0}_{L^\infty_TL^\infty(\Omega)} + \delta\gamma_0 + \gamma_{1},
\\
&\vdots
\end{align*}
where for each $n$, we may write $\gamma_{n-1}$ by the previous step. We observe that $\gamma_{n-1}$ is decreasing. Indeed, a new term is added at each step, but it is kept small by the smallness-condition and by restricting the time $T$, resulting in a geometric series. Thus, continuing while restricting $\gamma_{n-1}<1$ for each $n\geq1$ implies that $\gamma_0 \geq \gamma_1 \geq \dots \geq \gamma_{n-2} \geq \gamma_{n-1}$. Iterating over $n$ in \eqref{eq:beforeiteration} therefore yields
\begin{align*}
\norm{u^n}_{L^\infty_TL^\infty(\Omega)} \leq \hat{c}\sum_{k=1}^{n-1}\delta^k +\delta^n\norm{u^0}_{L^\infty_TL^\infty(\Omega)} +\gamma_0\sum_{k=1}^{n-1}\delta^k
&\leq( \hat{c}+\gamma_0)\sum_{k=1}^{\infty}\delta^k.
\end{align*}
As a result
\begin{align}\label{eq:ThisONE}
\norm{u^n}_{L^\infty_TL^\infty(\Omega)} \leq c(\norm{(f_0,f_N,\beta_0,u^0)}_{L^\infty_TL^\infty(\Omega) \times L^\infty_TL^\infty(\Gamma_N)\times L^\infty(\Gamma_C)  \times L^\infty_T ( H^1(\Omega) \cap L^\infty(\Omega)) })
\end{align}
for $\lim_{n\rightarrow\infty} \gamma_{n-1}=: \gamma>0$. By extension, it implies that there is a positive time $T$ independent of $n$ such that $u^n$ is uniformly bounded in $L^\infty_TL^\infty(\Omega)$. To conclude, it only remains to show that $u^n$ is uniformly bounded in $L^\infty_T H^1_0(\Omega)$. In particular, combining the estimates \eqref{eq:est_u_n_V} and \eqref{eq:beta_est_not_squared} implies
\begin{align*}
\norm{u^n}_{L^\infty_T H^1(\Omega)}     &\leq c(1 +   \norm{f_0}_{L^\infty_TL^\infty(\Omega)} + \norm{f_N}_{L^\infty_TL^\infty(\Gamma_N)}) +\frac{c_{2\varphi} c_0}{\mu_\ast} \norm{u^{n-1}}_{L^\infty_T H^1(\Omega)} \\ 
&+\frac{c_{1\varphi} c_0}{\mu_\ast} (\norm{\beta_0}_{L^\infty(\Gamma_C)}  + T^\alpha R(n-1)) \\
&+\frac{c_{1\phi}c_0}{\mu_\ast}(\norm{\beta_0}_{L^\infty(\Gamma_C)}^2 + 2\norm{\beta_0}_{L^\infty(\Gamma_C)}T^\alpha R(n-1) +T^{2\alpha} R(n-1)^2) \norm{u^{n-1}}_{L^\infty_T H^1(\Omega)}
\end{align*}
for $\alpha \geq 1$, where $R(n-1)$ is defined in \eqref{eq:R}. We observe that $R(n-1)$ only depends on $\norm{u^{n-1}}_{L^\infty_T L^\infty(\Omega)}$ and the given data. We therefore utilize \eqref{eq:ThisONE} to obtain
\begin{align*}
R (n-1) \leq c(\norm{(f_0,f_N,\beta_0,u^0)}_{L^\infty_TL^\infty(\Omega) \times L^\infty_TL^\infty(\Gamma_N)\times L^\infty(\Gamma_C)  \times L^\infty_T ( H^1(\Omega) \cap L^\infty(\Omega)) }).
\end{align*}
Consequently, there is a $c>0$ depending on the given data such that
\begin{align*}
\norm{u^n}_{L^\infty_T H^1(\Omega)}     &\leq c
+  (\frac{c_{2\varphi}c_0 +c_{1\phi} c_0\norm{\beta_0}_{L^\infty(\Gamma_C)}^2 }{\mu_\ast}  + c T^\alpha)  \norm{u^{n-1}}_{L^\infty_T H^1(\Omega)} 
\end{align*}
for $\alpha \geq 1$. From the smallness assumption \eqref{eq:smallness}, we choose $T>0$ so that
\begin{align}\label{eq:T_less_than_1}
\frac{c_{2\varphi}c_0 +c_{1\phi} c_0\norm{\beta_0}_{L^\infty(\Gamma_C)}^2 }{\mu_\ast} + cT^\alpha <1 
\end{align}
for $\alpha \geq 1$ and $c=c(\norm{(f_0,f_N,\beta_0,u^0)}_{L^\infty_TL^\infty(\Omega) \times L^\infty_TL^\infty(\Gamma_N)\times L^\infty(\Gamma_C)  \times L^\infty_T ( H^1(\Omega) \cap L^\infty(\Omega)) })$. Utilizing Lemma \ref{lemma:estn} gives us the desired estimate.
\end{proof}
With the estimate \eqref{eq:estVn1} in Proposition \ref{prop:n}, we may follow the same procedure as for $n=1$ (Lemma \ref{lemma:lambda_fixedpoint}) to conclude that $\{(u^n,\beta^n)\}_{n\geq 1} \subset L^\infty_T ( H^1_0(\Omega) \cap L^\infty(\Omega)) \times C([0,T];L^2(\Gamma_C)) \cap L^\infty_T L^\infty(\Gamma_C)$ indeed is the approximated solution to \eqref{eq:u_n}-\eqref{eq:beta_eq}. In addition, we observe that \eqref{eq:est_beta_useful} only depends on $\norm{u^n}_{L^\infty_T L^\infty(\Omega)}$ and the initial data, so we may use the estimate in Proposition \ref{prop:n} to conclude that
\begin{align}\label{eq:est_beta_N}
\norm{\beta^n}_{L^\infty_T L^\infty(\Gamma_C)} &\leq  c(\norm{(f_0,f_N,\beta_0,u^0)}_{L^\infty_TL^\infty(\Omega) \times L^\infty_TL^\infty(\Gamma_N)\times L^\infty(\Gamma_C)  \times L^\infty_T ( H^1(\Omega) \cap L^\infty(\Omega)) })
\end{align}
for some $T>0$ independent of $n$.
\subsection*{Step 5.2 \textit{(Convergence of the approximated solutions)}.}
We will next show that $\{(u^n,\beta^n)\}_{n\geq 1}$ is a Cauchy sequence in $L^\infty_T  H^1_0(\Omega) \times C([0,T];L^2(\Gamma_C))$. This is summarized in the proposition below.
\begin{prop}\label{prop:cauchy}
Let $\beta^0=\beta_0 \in L^\infty(\Gamma_C)$ and $u^0 \in L^\infty_T ( H^1_0(\Omega) \cap L^\infty(\Omega))$. Under the assumptions \hyperref[assumptionvarphi]{$(H1)$}-\hyperref[assumptionH]{$(H3)$} and \eqref{eq:assumptions_mu}-\eqref{eq:assumptionondata}, let $\{(u^n,\beta^n)\}_{n\geq 1} \subset L^\infty_T ( H^1_0(\Omega) \cap L^\infty(\Omega)) \times C([0,T];L^2(\Gamma_C))\cap L^\infty_TL^\infty(\Gamma_C)$ be the solution to \eqref{eq:u_n}-\eqref{eq:beta_eq}. Then, $\{(u^n,\beta^n)\}_{n\geq 1}$ is a Cauchy sequence in $L^\infty_T  H^1_0(\Omega) \times C([0,T];L^2(\Gamma_C))$.
\end{prop}
\begin{proof}
Let $e_u^n = u^n-u^{n-1}$ and $e_\beta^n = \beta^n - \beta^{n-1}$. To begin with, we add the inequality \eqref{eq:u_n} at level $n$ with $v=u^{n-1}(t)$ and at level $n-1$ with $v=u^n(t)$, that is
\begin{align*}
\int_\Omega &\mu |\nabla e_u^n(t) |^2 \d x\\
&\leq  \int_{\Gamma_C} [\varphi(\beta^{n-1}(t),  u^{n-1}(t),  u^{n-1}(t))-\varphi(\beta^{n-1}(t), u^{n-1}(t), u^n(t))]\d \sigma \\
&+ \int_{\Gamma_C} [\varphi(\beta^{n-2}(t),  u^{n-2}(t),  u^{n}(t))-\varphi(\beta^{n-2}(t), u^{n-2}(t), u^{n-1}(t))] \d \sigma \\
&+\int_{\Gamma_C} [\phi(\beta^{n-1}(t),  u^{n-1}(t),  u^{n-1}(t))-\phi(\beta^{n-1}(t), u^{n-1}(t), u^n(t))]\d \sigma \\
&+ \int_{\Gamma_C} [\phi(\beta^{n-2}(t),  u^{n-2}(t),  u^{n}(t))-\phi(\beta^{n-2}(t), u^{n-2}(t), u^{n-1}(t))] \d \sigma 
\end{align*}
for a.e. $t\in(0,T)$. Applying the lower bound of $\mu$ \eqref{eq:assumptionmu}, \hyperref[assumptionvarphi]{$(H1)$}$(ii)$, and \hyperref[assumptiophi]{$(H2)$}$(ii)$ implies
\begin{align*}
\mu_\ast  \int_\Omega |\nabla e_u^n(t)|^2 \d x 
&\leq c_{1\varphi}\int_{\Gamma_C}  |e_\beta^{n-1}(t)||e_u^n(t)| \d \sigma + c_{2\varphi}\int_{\Gamma_C} | e_u^{n-1}(t)| |e_u^n(t)| \d \sigma\\
&+ c_{1\phi}\int_{\Gamma_C} |\beta^{n-1}(t)|^2 | e_u^{n-1}(t)| |e_u^n(t)| \d \sigma\\
&+ c_{2\phi}\int_{\Gamma_C} |\beta^{n-2}(t)||u^{n-2}(t)| | e_\beta^{n-1}(t)| |e_u^n(t)| \d \sigma\\
&+ c_{3\phi}\int_{\Gamma_C} |\beta^{n-1}(t)||u^{n-2}(t)| | e_\beta^{n-1}(t)| |e_u^n(t)| \d \sigma
\end{align*}
for a.e. $t\in(0,T)$.	Consequently, by the Cauchy-Schwarz inequality and Proposition \ref{prop:trace} (with $\epsilon=1$) and \ref{prop:traceinfty}, and the Poincar\'{e} inequality
\begin{align*}
\norm{e_u^n}_{L^\infty_T H^1(\Omega)}^2 &\leq   \frac{c_{1\varphi}c_0}{\mu_\ast} \norm{e_\beta^{n-1}}_{L^\infty_TL^2(\Gamma_C)}\norm{e_u^n}_{L^\infty_T H^1(\Omega)}+  \frac{c_{2\varphi}c_0}{\mu_\ast}\norm{e_u^{n-1}}_{L^\infty_T H^1(\Omega)}\norm{e_u^n}_{L^\infty_T H^1(\Omega)}\\ \notag
&+\frac{c_{1\phi}c_0}{\mu_\ast}\norm{\beta^{n-1}}_{L^\infty_TL^\infty(\Gamma_C)}^2\norm{e_u^{n-1}}_{L^\infty_T H^1(\Omega)}\norm{e_u^n}_{L^\infty_T H^1(\Omega)} \\ \notag
&+\frac{c_{2\phi}c_0}{\mu_\ast}\norm{\beta^{n-1}}_{L^\infty_TL^\infty(\Gamma_C)}\norm{u^{n-2}}_{L^\infty_TL^\infty(\Omega)}\norm{e_\beta^{n-1}}_{L^\infty_TL^2(\Gamma_C)}\norm{e_u^n}_{L^\infty_T H^1(\Omega)} \\
&+\frac{c_{3\phi}c_0}{\mu_\ast}\norm{\beta^{n-2}}_{L^\infty_TL^\infty(\Gamma_C)}\norm{u^{n-2}}_{L^\infty_TL^\infty(\Omega)}\norm{e_\beta^{n-1}}_{L^\infty_TL^2(\Gamma_C)} \norm{e_u^n}_{L^\infty_T H^1(\Omega)},\notag
\end{align*}
since $\mu_\ast >0$. From the estimate \eqref{eq:beta_est_not_squared} (replacing $n-1$ with $n$) and \eqref{eq:est_beta_N}, we obtain
\begin{align}\label{eq:estused}
&\hspace{-0.1cm}\norm{e_u^n}_{L^\infty_T H^1(\Omega)}\\ \notag
&\leq   \frac{c_{1\varphi}c_0}{\mu_\ast} \norm{e_\beta^{n-1}}_{L^\infty_TL^2(\Gamma_C)}+  \frac{c_{2\varphi}c_0}{\mu_\ast}\norm{e_u^{n-1}}_{L^\infty_T H^1(\Omega)}\\ \notag
&+\frac{c_{1\phi}c_0}{\mu_\ast}(\norm{\beta_0}_{L^\infty(\Gamma_C)}^2 + 2\norm{\beta_0}_{L^\infty(\Gamma_C)} T^{\alpha} R(n-1)+T^{2\alpha} R(n-1)^2)\norm{e_u^{n-1}}_{L^\infty_T H^1(\Omega)} \\ \notag
&+c(\norm{\beta^{n-1}}_{L^\infty_T L^\infty(\Gamma_C)}+\norm{\beta^{n-2}}_{L^\infty_T L^\infty(\Gamma_C)})\norm{u^{n-2}}_{L^\infty_TL^\infty(\Omega)}\norm{e_\beta^{n-1}}_{L^\infty_TL^2(\Gamma_C)} \notag
\end{align}
for $\alpha \geq 1$ and $R$ defined in \eqref{eq:R}. Next, we subtract \eqref{eq:beta_eq} for the levels $n$ and $n-1$, and then apply \hyperref[assumptionH]{$(H3)$}$(ii)$ to obtain
\begin{align*}
|e_\beta^n(t)|
&\leq \int_0^t ( c_{1 \beta}|\beta^n(s)||u^n(s)|+c_{2 \beta}|\beta^n(s)||u^{n-1}(s)|)|e^n_u(s)| \d s  \\
&+\int_0^t   c_{3\beta} |u^{n-1}(s)|^2|e^n_\beta(s)| \d s
\end{align*}
for a.e. $t\in(0,T)$. By H\"{o}lder's inequality and Proposition \ref{prop:trace} (for $\epsilon=1$) and \ref{prop:traceinfty}, we have
\begin{align*}
\norm{e_\beta^n(t)}_{L^2(\Gamma_C)} &\leq \int_0^t(c_{1 \beta}c_0 \norm{\beta^n(s)}_{L^\infty(\Gamma_C)} + c_{2 \beta}c_0 \norm{\beta^n(s)}_{L^\infty(\Gamma_C)} \norm{u^{n-1}(s)}_{L^\infty(\Omega)})\norm{e_u^n(s)}_{H^1(\Omega)} \d s \\
&+ \int_0^t c_{3\beta}c_0 \norm{u^{n-1}(s)}_{L^\infty(\Omega)}^2 \norm{e_\beta^n(s)}_{L^2(\Gamma_C)} \d s
\end{align*}
for a.e. $t\in(0,T)$. Again, applying H\"{o}lder's inequality reads 
\begin{align*}
\norm{e_\beta^n(t)}_{L^2(\Gamma_C)}
&\leq cT \norm{\beta^n}_{L^\infty_T L^\infty(\Gamma_C)}(1 +\norm{u^{n-1}}_{L^\infty_TL^\infty(\Omega)})\norm{e_u^n}_{L^\infty_T H^1(\Omega)} \\
& + c_{3\beta} \norm{u^{n-1}}_{L^\infty_TL^\infty(\Omega)}^2\int_0^t\norm{e_\beta^n(s)}_{L^2(\Gamma_C)} \d s
\end{align*}
for a.e. $t\in(0,T)$. Utilizing a standard Gr\"{o}nwall's argument  gives us
\begin{align*}
\norm{e_\beta^n(t)}_{L^2(\Gamma_C)} 
\leq& cT\norm{\beta^{n}}_{L^\infty_T L^\infty(\Gamma_C)}(1 +\norm{u^{n-1}}_{L^\infty_TL^\infty(\Omega)})\\
&\times(1+ cT \norm{u^{n-1}}_{L^\infty_TL^\infty(\Omega)}^2\mathrm{e}^{cT\norm{u^{n-1}}_{L^\infty_TL^\infty(\Omega)}^2})\norm{e_u^n}_{L^\infty_T H^1(\Omega)} 
\end{align*}
for a.e. $t\in(0,T)$. Choosing $T >0$ as in \eqref{eq:T} yields
\begin{align}\label{eq:est_beta_n}
\norm{e_\beta^n}_{L^\infty_TL^2(\Gamma_C)} 
&\leq cT\norm{\beta^{n}}_{L^\infty_T L^\infty(\Gamma_C)}(1 +\norm{u^{n-1}}_{L^\infty_TL^\infty(\Omega)})\norm{e_u^n}_{L^\infty_T H^1(\Omega)} .
\end{align}
Inserting \eqref{eq:est_beta_n} at level $n-1$  into \eqref{eq:estused} while keeping in mind Proposition \ref{prop:estn} gives us
\begin{align*}
&\norm{e_u^n}_{L^\infty_T H^1(\Omega)}  \leq \frac{c_{2\varphi}c_0+c_{1\phi}c_0\norm{\beta_0}_{L^\infty(\Gamma_C)}^2}{\mu_\ast}\norm{e_u^{n-1}}_{L^\infty_T H^1(\Omega)} +cT^\alpha  \norm{e_u^{n-1}}_{L^\infty_T H^1(\Omega)} 
\end{align*}
for $\alpha \geq 1$ and $c=c(\norm{(f_0,f_N,\beta_0,u^0)}_{L^\infty_TL^\infty(\Omega) \times L^\infty_TL^\infty(\Gamma_N)\times L^\infty(\Gamma_C)  \times L^\infty_T (H^1(\Omega) \cap L^\infty(\Omega)) })$. For $\alpha \geq 1$, we define
\begin{align*}
B&:=  \frac{c_{2\varphi}c_0 + c_{1\phi}c_0\norm{\beta_0}_{L^\infty(\Gamma_C)}^2 }{\mu_\ast} \\
&+c(\norm{(f_0,f_N,\beta_0,u^0)}_{L^\infty_TL^\infty(\Omega) \times L^\infty_TL^\infty(\Gamma_N)\times L^\infty(\Gamma_C)  \times L^\infty_T (H^1(\Omega) \cap L^\infty(\Omega)) })T^\alpha.
\end{align*}
From the smallness-condition \eqref{eq:smallness} and choosing $T>0$ in the same way as in \eqref{eq:T_less_than_1}, we have that $B<1$. Iterating over $n \in \N$ yields
\begin{align}\label{eq:eun}
\norm{e_u^n}_{L^\infty_T H^1(\Omega)} \leq cB^n .
\end{align}
This results in $\lim_{n\rightarrow \infty}cB^{n} = 0 $. Hence, passing the limit $n\rightarrow \infty$ in \eqref{eq:eun} implies that $\{u^n\}_{n\geq 1} \subset L^\infty_T H^1_0(\Omega)$ is a Cauchy sequence. As a consequence, it follows by \eqref{eq:est_beta_N}, \eqref{eq:est_beta_n}, and Proposition \ref{prop:n} that $\{\beta^n\}_{n\geq1} \subset C([0,T];L^2(\Gamma_C))$ also is a Cauchy sequence. 
\end{proof}

\subsection*{Step 5.3 \textit{(Passing to the limit in \eqref{eq:u_n}-\eqref{eq:beta_eq})}.}
From Proposition \ref{prop:cauchy}, it follows that as $n\rightarrow \infty$
\begin{align*}
u^n &\rightarrow u   \text{ strongly in } L^\infty_T H^1_0(\Omega) &\text{and}&	&\beta^n &\rightarrow \beta   \text{ strongly in } C([0,T],L^2(\Gamma_C)).
\end{align*}
Consequently, by Lemma \ref{lemma:strongtoae} we have, up to a subsequence,
\begin{subequations}\label{eq:u_beta_convererge}
\begin{align}\label{eq:u_convererge}
u^n(t) &\rightarrow u(t)   \text{ strongly in } H^1_0(\Omega) &\text{ for a.e. } t\in (0,T)\\
\beta^n(t) &\rightarrow \beta(t)  \text{ strongly in } L^2(\Gamma_C)&  \text{ for a.e. } t\in (0,T)
\label{eq:beta_convererge}
\end{align}
\end{subequations}
as $n\rightarrow \infty$. It follows that
\begin{subequations}\label{eq:u_beta_boundsss}
\begin{align}\label{eq:u_beta_bounded_VL2}
\norm{u}_{L^\infty_T H^1(\Omega)} &\leq c &\text{and}& &\norm{\beta}_{L^\infty_TL^2(\Gamma_C)} &\leq c.	
\end{align}
As a consequence of the Banach-Alaoglu theorem (see, e.g., \cite[Corollary 3.30]{Brezis2011}), we have by Proposition \ref{prop:n} and \eqref{eq:est_beta_N} that 
\begin{align}\label{eq:u_beta_bounded}
\norm{u}_{L^\infty_TL^\infty(\Omega)} &\leq c &\text{and}& &\norm{\beta}_{L^\infty_TL^\infty(\Gamma_C)} &\leq c.	
\end{align}
\end{subequations}
We are now in a position to pass the limit in \eqref{eq:u_n}-\eqref{eq:beta_eq}). From \eqref{eq:u_convererge} and \eqref{eq:assumptionmu}, we directly have as $n\rightarrow \infty$ that
\begin{align*}
\int_\Omega \mu \nabla u^n (t) \cdot \nabla (v-u^n(t)) \d x  \rightarrow \int_\Omega \mu \nabla u(t) \cdot \nabla (v-u(t)) \d x 
\end{align*}
for all $v\in H^1_0(\Omega)$ and a.e. $t\in (0,T)$. Next, we have by \hyperref[assumptionvarphi]{$(H1)$}$(ii)$ and Lemma \ref{lemma:calculation} that
\begin{align*}
&\varphi(\beta^n(t), u^n(t),  v)-\varphi(\beta^n(t), u^n(t), u^n(t)) - \varphi(\beta(t), u(t),  v) +\varphi(\beta(t), u(t), u(t))\\
&= 	\varphi(\beta^n(t), u^n(t),  v)-\varphi(\beta^n(t), u^n(t), u^n(t)) + \varphi(\beta(t), u(t), u^n(t)) - \varphi(\beta(t), u(t),  v) \\
&+\varphi(\beta(t), u(t), u(t)) -\varphi(\beta(t), u(t), u^n(t)) \\
&\leq  c_{1\varphi} |\beta^n(t)-\beta(t)| |u^n(t)-v| +c_{2\varphi}|u^n(t)-u(t)|  |u^n(t)-v| \\
&+ c_{0\varphi} |u^n(t)-u(t)| + c_{1\varphi} |\beta(t)| |u^n(t)-u(t)| +c_{2\varphi} |u(t)|  |u^n(t)-u(t)|
\end{align*}
for all $v\in H^1_0(\Omega)$, $n\in\N$. Consequently, H\"{o}lder's inequality and Proposition \ref{prop:trace} with $\epsilon=1$ implies
\begin{align*}
&\int_{\Gamma_C} [\varphi(\beta^n(t), u^n(t),  v)-\varphi(\beta^n(t), u^n(t), u^n(t))] \d \sigma - \int_{\Gamma_C}[\varphi(\beta(t), u(t),  v)-\varphi(\beta(t), u(t), u(t))] \d \sigma\\
& \leq  c_{1\varphi}c_0 \norm{\beta^n(t)-\beta(t)}_{L^2(\Gamma_C)} \norm{u^n(t)-v}_{H^1(\Omega)} +c_{2\varphi}c_0\norm{u^n(t)-u(t)}_{H^1(\Omega)}  \norm{u^n(t)-v}_{H^1(\Omega)} \\
&+ c_{0\varphi}c_0 \norm{u^n(t)-u(t)}_{H^1(\Omega)} + c_{1\varphi}c_0 \norm{\beta}_{L^\infty_TL^2(\Gamma_C)} \norm{u^n(t)-u(t)}_{H^1(\Omega)} \\
&+c_{2\varphi}c_0 \norm{u}_{L^\infty_T H^1(\Omega)}\norm{u^n(t)-u(t)}_{H^1(\Omega)}
\end{align*}
for all $v\in H^1_0(\Omega)$, $n\in\N$, and a.e. $t\in (0,T)$. We may conclude by \eqref{eq:u_convererge}-\eqref{eq:beta_convererge}, Proposition \ref{prop:n}, and \eqref{eq:u_beta_bounded_VL2}.
\\
\indent Similarly, from \hyperref[assumptiophi]{$(H2)$}$(ii)$ and Lemma \ref{lemma:calculation}, we have that
\begin{align*}
&\phi(\beta^n(t), u^n(t),  v)-\phi(\beta^n(t), u^n(t), u^n(t)) - \phi(\beta(t), u(t),  v) +\phi(\beta(t), u(t), u(t))\\
&= 	\phi(\beta^n(t), u^n(t),  v)-\phi(\beta^n(t), u^n(t), u^n(t)) + \phi(\beta(t), u(t), u^n(t)) - \phi(\beta(t), u(t),  v) \\
&+\phi(\beta(t), u(t), u(t)) -\phi(\beta(t), u(t), u^n(t)) \\
&\leq  c_{1\phi} |\beta^n(t)|^2 |u^n(t)-u(t)| |u^n(t)-v| +c_{2\phi}|\beta(t)||u(t)||\beta^n(t)-\beta(t)|  |u^n(t)-v| \\
&+ c_{3\phi}|\beta^n(t)||u(t)||\beta^n(t)-\beta(t)| |u^n(t)-v| + c_{1\phi} |\beta(t)|^2|u(t)||u^n(t)-u(t)|
\end{align*}
for all $v\in H^1_0(\Omega)$, $n\in\N$. Consequently, H\"{o}lder's inequality and Proposition \ref{prop:trace} (with $\epsilon=1$) and \ref{prop:traceinfty} implies
\begin{align*}
&\int_{\Gamma_C} [\phi(\beta^n(t), u^n(t),  v)-\phi(\beta^n(t), u^n(t), u^n(t))] \d \sigma- \int_{\Gamma_C}[\phi(\beta(t), u(t),  v)-\phi(\beta(t), u(t), u(t))] \d \sigma\\
& \leq  c_{1\phi}c_0  \norm{\beta^n}_{L^\infty_TL^\infty(\Gamma_C)}^2\norm{u^n(t)-u(t)}_{H^1(\Omega)}  \norm{u^n(t)-v}_{H^1(\Omega)} \\
&+
c_{2\phi}c_0 \norm{\beta}_{L^\infty_TL^\infty(\Gamma_C)}\norm{u}_{L^\infty_TL^\infty(\Omega)}\norm{\beta^n(t)-\beta(t)}_{L^2(\Gamma_C)} \norm{u^n(t)-v}_{H^1(\Omega)} \\
&+c_{3\phi}c_0\norm{\beta^n}_{L^\infty_TL^\infty(\Gamma_C)}\norm{u}_{L^\infty_TL^\infty(\Omega)}\norm{\beta^n(t)-\beta(t)}_{L^2(\Gamma_C)}\norm{u^n(t)-v}_{H^1(\Omega)} \\
&+ c_{1\phi}c_0\norm{\beta}_{L^\infty_TL^\infty(\Gamma_C)}\norm{u}_{L^\infty_TL^\infty(\Omega)} \norm{u^n(t)-u(t)}_{H^1(\Omega)}
\end{align*}
for all $v\in H^1_0(\Omega)$, $n\in\N$, and a.e. $t\in (0,T)$.  Concluding by \eqref{eq:u_beta_convererge}, \eqref{eq:u_beta_bounded}, Proposition \ref{prop:n}, and \eqref{eq:est_beta_N}. For the terms with $f_0$ and $f_N$ in \eqref{eq:u_n}, we may utilize the Cauchy–Schwarz inequality and \eqref{eq:assumptionondata} together with \eqref{eq:u_convererge}. For the term on $\Gamma_N$, we also need Proposition \ref{prop:trace} for $\epsilon=1$.
\\
\indent Lastly, we consider the equation \eqref{eq:beta_eq}. We first note by Lemma \ref{lemma:calculation}, Proposition \ref{prop:traceinfty} and \ref{prop:n}, and \eqref{eq:est_beta_N} that
\begin{align*}
| H(\beta^n(s), u^n(s))| \leq  c_{0\beta} +  c_{3\beta}\norm{u^n}_{L^\infty_T L^\infty(\Omega)}^2\norm{\beta^n}_{L^\infty_T L^\infty(\Gamma_C)} \leq c
\end{align*}
for a.e. $s\in (0,t)\subset (0,T)$. In addition, we have by \hyperref[assumptionH]{$(H3)$}$(ii)$ that
\begin{align}\label{eq:H_est}
&\norm{ H(\beta^n(s), u^n(s))-  H(\beta(s), u(s))}_{L^2(\Gamma_C)} \\
&\leq c( \norm{\beta^n}_{L^\infty_T L^\infty(\Gamma_C)}\norm{u^n}_{L^\infty_T \notag L^\infty(\Omega)}|+\norm{\beta^n}_{L^\infty_T L^\infty(\Gamma_C)} \norm{u}_{L^\infty_T L^\infty(\Omega)})\norm{u^n(s)-u(s)}_{H^1(\Omega)} \\
& + c\norm{u}_{L^\infty_TL^\infty(\Omega)}^2 \norm{\beta^n(s)-\beta(s)}_{L^2(\Gamma_C)} \notag
\end{align}  
for a.e. $s\in (0,t)\subset (0,T)$. Combining Proposition \ref{prop:n}, \eqref{eq:est_beta_N}, and \eqref{eq:u_beta_convererge}-\eqref{eq:u_beta_boundsss}, we may conclude by the dominated convergence theorem.\\
\indent Thus, passing the limit  $n\rightarrow \infty$ in \eqref{eq:u_n}-\eqref{eq:beta_eq} gives us that $(u,\beta) \in L^\infty_T (H^1_0(\Omega) \cap L^\infty(\Omega)) \times C([0,T];L^2(\Gamma_C))\cap L^\infty_TL^\infty(\Gamma_C)$ indeed is a solution to \eqref{eq:eq1}.

\subsection*{Step 6 \textit{(Uniqueness)}.}
It only remains to prove that the solution of \eqref{eq:eq1} is unique. Let $(\beta_1,u_1),(\beta_2,u_2)$ be two solutions to \eqref{eq:eq1}. We choose $v=u_2$ and $v=u_1$, respectively:
\begin{align*}
\int_\Omega &\mu \nabla(u_1(t)-u_2(t))\cdot \nabla(u_1(t)-u_2(t)) \d x\\
&\leq  \int_{\Gamma_C} [\varphi(\beta_1(t),  u_1(t), u_1(t))-\varphi(\beta_1(t), u_1(t), u_2(t))]\d \sigma \\
&+ \int_{\Gamma_C}[\varphi(\beta_2(t),  u_2(t), u_2(t))-\varphi(\beta_2(t),u_2(t), u_1(t))] \d \sigma \\
&+\int_{\Gamma_C} [\phi(\beta_1(t),  u_1(t), u_1(t))-\phi(\beta_1(t), u_1(t), u_2(t))]\d \sigma \\
&+ \int_{\Gamma_C}[\phi(\beta_2(t),  u_2(t), u_2(t))-\phi(\beta_2(t),u_2(t), u_1(t))] \d \sigma 
\end{align*}
for a.e. $t\in(0,T)$. As a result of the lower bound of $\mu$ \eqref{eq:assumptionmu}, \hyperref[assumptionvarphi]{$(H1)$}$(ii)$, \hyperref[assumptiophi]{$(H2)$}$(ii)$, H\"{o}lder's inequality, Proposition \ref{prop:trace} (for $\epsilon=1$) and \ref{prop:traceinfty}, and the Poincar\'{e} inequality, we have
\begin{align*}
& \norm{u_1(t)-u_2(t)}_{H^1(\Omega)}^2 \\
&\leq \frac{c_{1\varphi}c_0}{\mu_\ast}\norm{\beta_1(t) - \beta_2(t)}_{L^2(\Gamma_C)}\norm{u_1(t)-u_2(t)}_{H^1(\Omega)} +  \frac{c_{2\varphi}c_0}{\mu_\ast}\norm{u_1(t)-u_2(t)}_{H^1(\Omega)}^2\\
&+  \frac{c_{1\phi}c_0}{\mu_\ast} \norm{\beta_1}_{L^\infty_TL^\infty(\Gamma_C)}^2\norm{u_1(t)-u_2(t)}_{H^1(\Omega)}^2  \\
&+ \frac{c_{2\phi}c_0}{\mu_\ast} \norm{\beta_2}_{L^\infty_TL^\infty(\Gamma_C)}\norm{u_2}_{L^\infty_TL^\infty(\Omega)}\norm{\beta_1(t) - \beta_2(t)}_{L^2(\Gamma_C)} \norm{u_1(t)-u_2(t)}_{H^1(\Omega)} \\
&+ \frac{c_{3\phi}c_0}{\mu_\ast} \norm{\beta_1}_{L^\infty_TL^\infty(\Gamma_C)}\norm{u_2}_{L^\infty_TL^\infty(\Omega)} \norm{\beta_1(t) - \beta_2(t)}_{L^2(\Gamma_C)} \norm{u_1(t)-u_2(t)}_{H^1(\Omega)}
\end{align*}
for a.e. $t\in (0,T)$. Now, from \eqref{eq:eq1b}, Lemma \ref{lemma:calculation}, and Proposition \ref{prop:traceinfty}, we deduce
\begin{align*}
\norm{\beta_j(t)}_{L^\infty(\Gamma_C)}  &\leq \norm{\beta_0}_{L^\infty(\Gamma_C)} + \int_0^t \norm{H(\beta_j(s), u_j(s))}_{L^\infty(\Gamma_C)} \d s\\
&\leq \norm{\beta_0}_{L^\infty(\Gamma_C)} +c_{0\beta}c_0T + c_{3\beta}c_0\norm{u_j}^2_{L^\infty_TL^\infty(\Omega)} \int_0^t  \norm{\beta_j(s)}_{L^\infty(\Gamma_C)}  \d s
\end{align*}
for a.e. $t\in (0,T)$ and $j=1,2$. An application of Gr\"{o}nwall's inequality implies
\begin{align*}
\norm{\beta_j}_{L^\infty_TL^\infty(\Gamma_C)}  
&\leq \norm{\beta_0}_{L^\infty(\Gamma_C)}(1+cT\norm{u_j}^2_{L^\infty_TL^\infty(\Omega)}\mathrm{e}^{cT\norm{u_j}^2_{L^\infty_TL^\infty(\Omega)}}) \\
&+ cT (1+cT\norm{u_j}^2_{L^\infty_TL^\infty(\Omega)}\mathrm{e}^{cT\norm{u_j}^2_{L^\infty_TL^\infty(\Omega)}}) .
\end{align*}
From  \eqref{eq:u_beta_bounded}, we can choose $T>0$ as in \eqref{eq:T} so that
\begin{align*}
\norm{\beta_j}_{L^\infty_TL^\infty(\Gamma_C)}  
&\leq \norm{\beta_0}_{L^\infty(\Gamma_C)} + cT \norm{\beta_0}_{L^\infty(\Gamma_C)}\norm{u_j}^2_{L^\infty_TL^\infty(\Omega)} \\
&+ cT^\alpha (1+\norm{u_j}^2_{L^\infty_TL^\infty(\Omega)}) \\
&=: \norm{\beta_0}_{L^\infty(\Gamma_C)} + T^\alpha \tilde{R}
\end{align*}
for $\alpha \geq 1$ and $j=1,2$. We also subtract \eqref{eq:eq1b} for $\beta_1$ and $\beta_2$. This yields
\begin{align}\label{eq:this_Eq}
\norm{\beta_1(t) - \beta_2(t)}_{L^2(\Gamma_C)}  \leq \int_0^t \norm{ H(\beta_1(s), u_1(s))-  H(\beta_2(s), u_2(s))}_{L^2(\Gamma_C)} \d s
\end{align}
for a.e. $t\in (0,T)$. Utilizing \eqref{eq:H_est} with $u^n = u_1$, $u=u_2$, $\beta^n = \beta_1$, and $\beta = \beta_2$ implies
\begin{align*}
&\norm{ H(\beta_1(s), u_(s))-  H(\beta_2(s), u_2(s))}_{L^2(\Gamma_C)} \\
&\leq c( \norm{\beta_1}_{L^\infty_T L^\infty(\Gamma_C)}\norm{u_1}_{L^\infty_T L^\infty(\Omega)}+\norm{\beta_1}_{L^\infty_T L^\infty(\Gamma_C)} \norm{u_2}_{L^\infty_T L^\infty(\Omega)})\norm{u_1-u_2}_{L^\infty_T H^1(\Omega)} \\
& + c\norm{u_2}_{L^\infty_TL^\infty(\Omega)}^2 \norm{\beta_1(s)-\beta_2(s)}_{ L^2(\Gamma_C)}\\
&=: c_1\norm{u_1-u_2}_{L^\infty_T H^1(\Omega)} + c_2\norm{\beta_1(s)-\beta_2(s)}_{ L^2(\Gamma_C)}
\end{align*} 
for a.e. $s\in (0,t)\subset (0,T)$. Inserting the above estimate into \eqref{eq:this_Eq} gives us
\begin{align*}
\norm{\beta_1(t) - \beta_2(t)}_{L^2(\Gamma_C)} 
&\leq c_1T\norm{u_1-u_2}_{L^\infty_T H^1(\Omega)} + c_2\int_0^t\norm{\beta_1(s)-\beta_2(s)}_{L^2(\Gamma_C)} \d s
\end{align*}
for a.e. $t\in (0,T)$. From Gr\"{o}nwall's inequality, we obtain 
\begin{align*}
\norm{\beta_1 - \beta_2}_{L^\infty_T L^2(\Gamma_C)} 
&\leq c_1T\norm{u_1-u_2}_{L^\infty_T H^1(\Omega)} (1+c_2T\mathrm{e}^{c_2T}).
\end{align*}
From \eqref{eq:u_beta_bounded}, we may choose $T>0$ as in \eqref{eq:T}. This implies
\begin{align}\label{eq:beta12}
\norm{\beta_1 - \beta_2}_{L^\infty_T L^2(\Gamma_C)} 
&\leq cT\norm{u_1-u_2}_{L^\infty_T H^1(\Omega)}.
\end{align}
Gathering the estimates above reads
\begin{align*}
\norm{u_1-u_2}_{L^\infty_T H^1(\Omega)} &\leq cT\norm{u_1-u_2}_{L^\infty_T H^1(\Omega)} + \frac{c_{2\varphi}c_0}{\mu_\ast} \norm{u_1-u_2}_{L^\infty_T H^1(\Omega)}\\
&+  \frac{c_{1\phi}c_0}{\mu_\ast} (\norm{\beta_0}_{L^\infty(\Gamma_C)}^2 + 2\norm{\beta_0}_{L^\infty(\Gamma_C)}T^\alpha \tilde{R} + T^{2\alpha}\tilde{R}^2)\norm{u_1-u_2}_{L^\infty_T H^1(\Omega)}  \\
&+ cT(\norm{\beta_1}_{L^\infty_TL^\infty(\Gamma_C)}  + \norm{\beta_2}_{L^\infty_TL^\infty(\Gamma_C)} )\norm{u_2}_{L^\infty_TL^\infty(\Omega)}\norm{u_1-u_2}_{L^\infty_T H^1(\Omega)}
\end{align*}
for $\alpha \geq 1$. With the estimate \eqref{eq:u_beta_bounded}, we obtain
\begin{align*}
\norm{u_1-u_2}_{L^\infty_T H^1(\Omega)} \leq ( \frac{c_{2\varphi}c_0+c_{1\phi}c_0 \norm{\beta_0}_{L^\infty(\Gamma_C)}^2}{\mu_\ast}  + cT^\alpha )\norm{u_1-u_2}_{L^\infty_T H^1(\Omega)}. 
\end{align*}
with $c=c(\norm{(f_0,f_N,\beta_0,u^0)}_{L^\infty_TL^\infty(\Omega) \times L^\infty_TL^\infty(\Gamma_N)\times L^\infty(\Gamma_C)  \times L^\infty_T (H^1(\Omega) \cap L^\infty(\Omega)) })$ and $\alpha \geq 1$.  We now conclude that $u_1=u_2$ by choosing $T>0$ as in \eqref{eq:T_less_than_1}. Since $u_1=u_2$, we have from \eqref{eq:beta12} that $\beta_1=\beta_2$.

\subsection*{Step 7 \textit{(Proof of Theorem \ref{thm:mainresult})}.}
\begin{proof}[Proof of Theorem \ref{thm:mainresult}]
Combining Step 1-6 completes the proof of Theorem \ref{thm:mainresult}.
\end{proof}

\SkipTocEntry\section*{Acknowledgements}
\noindent
This research was supported by the VISTA program, the Norwegian Academy of Science and Letters, and Equinor. I would like to thank Mircea Sofonea for his helpful comments that improved a previous version of this manuscript.


\appendix

\section{Elastic antiplane contact problem with adhesion}\label{appendix:application}
\noindent  
We present a system of equations describing the deformation of an elastic cylinder in contact with an obstacle in the presence of frictional and adhesion processes between the bodies. The problem will be considered under the assumption that the cylinder is clamped on one part of the boundary.  We introduce three new evolution equations of the intensity of adhesion and a new static friction law, that is \hyperref[evol1]{$(E1)$}-\hyperref[evol3]{$(E3)$} and \eqref{eq:frictionlawused1}, respectively (see also Remark \ref{remark:comparion}).
\\
\indent In the mathematical treatment of this problem, we may simplify the usual frictional elasticity problem. The cylinder is assumed to be sufficiently long, let this be in the $x_3$-direction, so that the effects in the axial direction are negligible.
In other words, $\Omega \times (-\infty,\infty)$, where $\Omega \subset \R^2$ denote the reference configuration of the cross-section of the cylinder and  $\nu$ the outward normal on $\partial \Omega \times (-\infty,\infty)$ with $\partial \Omega = \Gamma_D \cup \Gamma_N \cup \Gamma_C$.
We are therefore only considering the displacement parallel to the $x_3$-direction, that is, $\mathbf{u} = (0,0,u)^T$ with $u=u(x_1,x_2,t)$. Similarly, for the body forces $\mathbf{f}_0$ and the surface traction $\mathbf{f}_N$ as illustrated in Figure \ref{fig:antipane}  (see also  \cite[Chapter 8]{Sofonea2009}). 
\begin{figure}[H]
\includegraphics[scale=0.84]{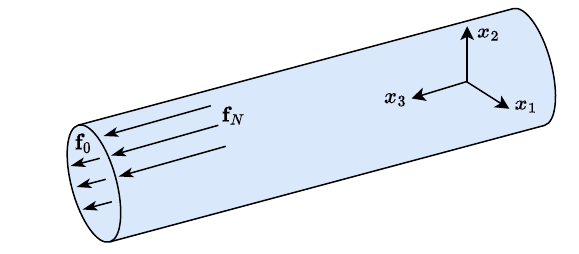}
\includegraphics[scale=0.5]{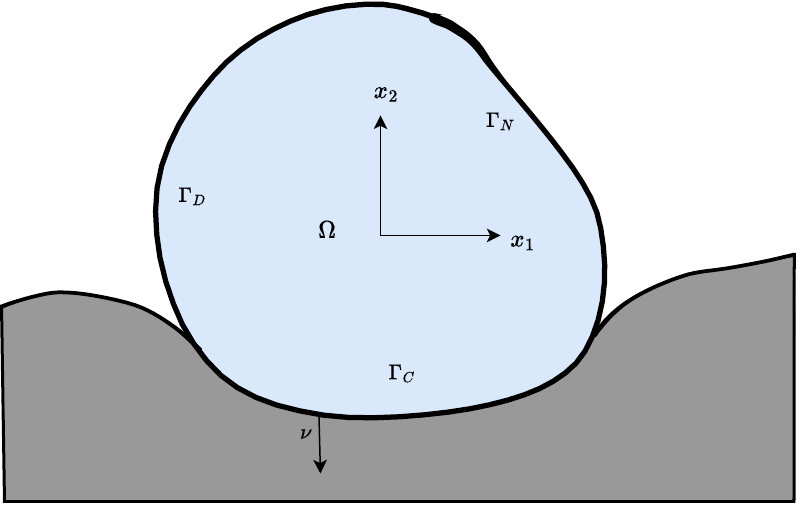}
\centering
\caption{An illustration of an antiplane shear problem.}
\label{fig:antipane}
\end{figure}\noindent
To be clear, we define
\begin{subequations}
\begin{align}\label{eq:surfacetension}
\mathbf{f}_0 &=(0,0, f_0)^T \  \ \text{ with } \ \ f_0 =f_0(x_1,x_2,t) : \Omega \times [0,T] \rightarrow \R,\\
\mathbf{f}_N &=(0,0, f_N)^T \ \ \text{ with } \ \ f_N =f_N(x_1,x_2,t) : \Gamma_N \times [0,T] \rightarrow \R.
\label{eq:and_traction}
\end{align} 
\end{subequations}
For the body \eqref{eq:surfacetension} and traction \eqref{eq:and_traction} forces,  the displacement of the cylinder is given by
\begin{align}\label{eq:u_z}
\mathbf{u} = (0,0,u)^T  \ \  \text{ for } \ \ u =u(x_1,x_2,t) : \Omega \times [0,T] \rightarrow \R.
\end{align}
We observe from \eqref{eq:u_z} that the infinitesimal deformation operator becomes
\begin{align}\label{eq:antiplaneshear} 
\varepsilon(\mathbf{u}) &= \frac{1}{2} (\nabla\mathbf{u} + (\nabla \mathbf{u})^T)
= \begin{bmatrix}
0 & 0 & \frac{1}{2}\frac{\partial u}{\partial x_1}\\
0 & 0 & \frac{1}{2}\frac{\partial u}{\partial x_2}\\
\frac{1}{2}\frac{\partial u}{\partial x_1} & \frac{1}{2}\frac{\partial u}{\partial x_2} & 0
\end{bmatrix}.
\end{align}
This deformation operator is referred to as antiplane shear. We will include some details on the derivation of the simplified friction problem in elasticity for the clarity of the reader (see also \cite{Sofonea2002}). For linearly elastic isometric material, the constitutive law is given by
\begin{align}\label{eq:isometricmaterial_nonred}
\sigma = \tilde{\mu} (\nabla \cdot \mathbf{u})\mathrm{Id} + 2 \mu \varepsilon,
\end{align}
where $\sigma$ is the stress tensor, $\mathrm{Id}$ is the identity map on $\R^3$, and $\mu$ and  $\tilde{\mu}$ are the Lam\'{e} parameters for nonhomogeneous material. Since $\mathbf{u}$ is independent of $x_3$, \eqref{eq:isometricmaterial_nonred} can be written as
\begin{align}\label{eq:isometricmaterial}
\sigma = 2\mu\varepsilon,
\end{align} 
where $\mu = \mu(x_1,x_2)$ in this setting. Considering  the momentum balance equation in the absence of acceleration 
\begin{align*}
-\nabla \cdot \sigma = \mathbf{f}_0,
\end{align*}
we may use \eqref{eq:surfacetension}, \eqref{eq:antiplaneshear}, and \eqref{eq:isometricmaterial} to reduce the equation to
\begin{align}\label{eq:simplfied_momentumbalance}
-\mathrm{div}(\mu \nabla u) = f_0.
\end{align}
Indeed, from  \eqref{eq:antiplaneshear} and \eqref{eq:isometricmaterial}, we have
\begin{align*}
\nabla \cdot \sigma =   2 \mu \nabla \cdot   \varepsilon + 2 \nabla \mu \cdot \varepsilon = \begin{bmatrix}
0\\
0 \\
\mu\frac{\partial^2 u}{\partial x_1^2} +\mu\frac{\partial^2 u}{\partial x_2^2} + \frac{\partial \mu}{\partial x_1} \frac{\partial u}{\partial x_1}+ \frac{\partial \mu}{\partial x_2}\frac{\partial u}{\partial x_2} 
\end{bmatrix}
= \begin{bmatrix}
0\\
0 \\
\mathrm{div}(\mu\nabla u)
\end{bmatrix}.
\end{align*}
Combining the above calculations with \eqref{eq:surfacetension} gives us the simplified momentum balance equation \eqref{eq:simplfied_momentumbalance}. 
\\
\indent Next, we need to complement the system with boundary conditions. We will first consider the Neumann and Dirichlet boundaries, that is, $\Gamma_N$ and $\Gamma_D$, respectively. The stress produced by the traction forces satisfy
\begin{align*}
\sigma\nu &= \mathbf{f}_N,
\end{align*}
where
\begin{align}\label{eq:sigmanu}
\sigma\nu &= (0,0,\mu \nabla u \cdot \nu)^T
\end{align}
and
\begin{align}\label{eq:normal}
\nu &= (\nu_1,\nu_2,0)^T \ \text{ with } \ \nu_1,\nu_2:\partial \Omega \rightarrow \R.
\end{align}
It therefore follows by \eqref{eq:and_traction} that the traction forces acting on $\Gamma_N$ must satisfy
\begin{align*}
\mu\nabla u \cdot \nu = f_N.
\end{align*}
Moreover, the cylinder is assumed to be clamped on $\Gamma_D\times (-\infty,\infty)$. This gives us the following condition on $\Gamma_D$
\begin{align*}
u= 0.
\end{align*}
Lastly, we consider the conditions on contact surface. These are described by the normal and tangential part of $\mathbf{u}$ and $\sigma$, which are, respectively, defined by
\begin{align*}
u_\nu &= \mathbf{u}\cdot \nu,  &\mathbf{u}_\tau  = \mathbf{u}-u_\nu \nu,& &\sigma_\nu = (\sigma\nu )\cdot \nu,&  &\text{and} &&	\sigma_\tau = \sigma\nu - \sigma_\nu\nu.
\end{align*}
The condition in the tangential direction is referred to as the friction law. A static friction law with adhesion is given by
\begin{align}\label{eq:frictionlaw}
\begin{dcases}
|\sigma_\tau - \lambda\beta^2 \mathbf{u}_\tau| \leq  g ,\\
\sigma_\tau- \lambda\beta^2  \mathbf{u}_\tau   = -g \frac{\mathbf{u}_\tau}{|\mathbf{u}_\tau|}, \quad \quad \text{ if } \mathbf{u}_\tau \neq 0,
\end{dcases}
\end{align}
where $g$ is the friction bound and $\lambda$ is a stiffness coefficient (see, e.g., \cite{Shillor2004,Sofonea2012}). Following \cite{Fremond1988},  $\beta$ is introduced as the measure of the intensity of adhesion (also referred to as the bonding field) (see also \cite{Shillor2004}). The physical constraint on the bonding field is $0\leq \beta \leq 1$; the bonds are active when $\beta = 1$ at a point on the contact surface, the bonds are severed if $\beta= 0$, and the bonds are partially active for $0<\beta <1$. As mentioned in Section \ref{sec:intro}, the adhesion processes are modelled through an ODE on the contact surface describing the evolution of the bonding field $\beta$. In our setting, we may simplify \eqref{eq:frictionlaw}. By \eqref{eq:u_z} and \eqref{eq:normal}, we first observe that $u_\nu=0$. In other words, the bodies stay in contact. With \eqref{eq:sigmanu}, this also implies that 
\begin{align*}
\mathbf{u}_\tau &=   (0,0,u) &\text{ and }&& \sigma_\tau  = (0,0,\mu \nabla u \cdot \nu).
\end{align*}
Consequently, \eqref{eq:frictionlaw} becomes the following friction law on $\Gamma_C$
\begin{align}\label{eq:frictionlawused1}
\begin{dcases}
|\mu\nabla u \cdot \nu - \lambda\beta^2 u| \leq  g ,\\
\mu\nabla u \cdot \nu - \lambda\beta^2 u  = -g \frac{u}{|u|}, \quad \quad \text{ if } u \neq 0,
\end{dcases}
\end{align}
where $g=g(|u|,\beta)$. The general framework introduced in this paper has applications of three new evolution equations of the bonding field (see Remark \ref{remark:comparion} for comparison to existing results):
\begin{table}[H]
\centering
\begin{tabular}{c||c||c}
Evolution equation $(E1)$\label{evol1}   &Evolution equation $(E2)$\label{evol2} & Evolution equation $(E3)$\label{evol3}  
\\
\hline  &&
\\
$\displaystyle\frac{\partial \beta}{\partial t} =  \mathcal{E}_D -   \lambda u^2  \beta $ &  $\displaystyle\frac{\partial \beta}{\partial t}=-(\lambda u^2 \beta -\mathcal{E}_D )_+$ &
$\displaystyle\frac{\partial \beta}{\partial t}  = -\lambda \frac{\beta}{1+\beta} u^2$
\end{tabular}
\end{table}\noindent
Here, $(\cdot)_+ = \max\{(\cdot),0\}$ and the coefficient $\mathcal{E}_D$ depends on the material (see, e.g., \cite[Section 4.2]{Shillor2004} for information on this quantity). The evolution equations \hyperref[evol1]{$(E1)$}-\hyperref[evol3]{$(E3)$} are related to \eqref{eq:eq1b}. The evolution equation \hyperref[evol2]{$(E2)$} is derived under the assumption that $\frac{\partial \beta}{\partial t} \leq 0$. Without this assumption, the general equation is given by \hyperref[evol1]{$(E1)$}. We refer the reader to \cite[Section 4.2]{Shillor2004} for more details. A truncated version of \hyperref[evol3]{$(E3)$} can be found in, e.g., \cite{Sofonea2002}. See Remark \ref{remark:comparion} for further comments. 
\\
\indent
Summarizing the above equations gives us the following problem with antiplane shear deformation.
\begin{prob}\label{prob:application}
Find the displacement $u: \Omega \times [0,T] \rightarrow \mathbb{R}$ and the bonding field $\beta : \Gamma_C \times [0,T] \rightarrow \mathbb{R}$ such that
\begin{subequations}
\begin{align}
-\mathrm{div}(\mu\nabla u(t)) &=  f_0(t) &  \text{ in }& \Omega \times (0,T)\\
u(t) &= 0 & \text{ on }& \Gamma_D \times (0,T)
\\
\mu\nabla u(t) \cdot \nu &= f_N(t) & \text{ on } &\Gamma_N \times (0,T)
\\
|\mu\nabla u(t) \cdot \nu- \lambda \beta^2 u| &\leq  g (|u(t)|, \beta(t))     & \text{ on }& \Gamma_C \times (0,T)
\label{eq:prob1_friction_eq2}
\\ \label{eq:prob1_friction_eq}
\mu\nabla u(t) \cdot \nu - \lambda \beta^2 u &= - g(|u(t)|, \beta(t))  \frac{u(t)}{|u(t)|}, \ \ \text{ if }u(t)\neq 0 & \text{ on }& \Gamma_C \times (0,T)\\
\frac{\partial \beta(t)}{\partial t} &=  H (\beta(t), u(t)) & \text{ on }& \Gamma_C \times (0,T)
\intertext{with the initial conditions}
\beta (0) &= \beta_{0}   &   \text{ on }& \Gamma_C,
\end{align}
\end{subequations}
where $H$ is given by either the right-hand side of  \hyperref[evol1]{$(E1)$},  \hyperref[evol2]{$(E2)$}, or  \hyperref[evol3]{$(E3)$}.
\end{prob}
Here, we investigate Problem \ref{prob:application} under the hypotheses \eqref{eq:assumptionmu}, \eqref{eq:assumptionondata},  and
\begin{align}\label{eq:assumption_g}
&g : \Gamma_C \times \mathbb{R} \times \mathbb{R} \rightarrow \mathbb{R} \text{ is  measurable, satisfies }\\ \notag
& |g(x,r_1,y_1) - g(x,r_2,y_2) | \leq   c_{1g}|y_1-y_2|+c_{2g}|r_1-r_2| \\ \notag
&\text{ for all } r_i,y_i\in \mathbb{R} \text{ with } i=1,2, \text{ a.e. } x\in \Gamma_C,\text{ and some } c_{1g},c_{2g} \geq 0,\\ \notag
& \text{ and there is a $c_{0g}\geq 0$ such that }  g(x,0,0) = c_{0g} \text{ for a.e. } x\in \Gamma_C,
\end{align}
\begin{align}\label{eq:assumption_EDLG}
\mathcal{E}_D &\in L^\infty(\Gamma_C) &\text{ and }& & \lambda \in L^\infty(\Gamma_C), \ \lambda \geq 0&. 
\end{align}
\begin{align}\label{eq:smallness_app}
\mu_\ast > c_{2g}c_0 + \norm{\lambda}_{L^\infty(\Gamma_C)}c_0 \norm{\beta_0}_{L^\infty(\Gamma_C)},
\end{align}
where $c_0 = c(|\Bar{\Omega}|)$ is a positive constant.
\begin{remark}
The condition on $\mu_\ast$ is a restriction on the elasticity of the material.
\end{remark}
\begin{remark}
Similar assumptions on $g$ are found in, e.g., \cite{Sofonea2018}.
\end{remark}
\begin{remark}
Neglecting $\beta^2$ in \eqref{eq:prob1_friction_eq2}-\eqref{eq:prob1_friction_eq}, the smallness assumption \eqref{eq:smallness_app} is reduced to 
\begin{align}\label{eq:smallness_app_updated}
\mu_\ast > c_{2g} c_0,
\end{align}
which is found in, e.g., \cite{Sofonea2002,Sofonea2020,Migorski2010}. 
\end{remark}
Assuming sufficient regularity, Problem \ref{prob:application} has the variational form 
\begin{subequations}
\begin{align}\label{eq:variational_form}
\int_\Omega& \mu \nabla u(t) \cdot \nabla (v-u(t)) \d x +	\int_{\Gamma_C} g(|u(t)|,\beta(t)) (|v|-|u(t)|) \d \sigma \\ \notag
&-\int_{\Gamma_C} \lambda \beta(t)^2 u(t) (v-u(t)) \d \sigma   \geq 	\int_\Omega f_0(t) (v-u(t)) \d x  + \int_{\Gamma_N} f_N(t) (v-u(t)) \d \sigma
\notag
\end{align}
for all $v\in H^1_0(\Omega)$ (defined in Section \ref{sec:funcspace}) and a.e. $t\in (0,T)$. Since we are interested in a mild solution to Problem \ref{prob:application}, we will study \eqref{eq:variational_form} coupled with
\begin{align}\label{eq:integralform_beta}
\beta(t) &=  \beta_0 + \int_0^t H (\beta(s), u(s)) \d s 
\end{align}
\end{subequations}
for all $x\in \Gamma_C$ and a.e. $t\in (0,T)$.
\begin{remark}
To obtain \eqref{eq:variational_form}, we refer the reader to, e.g., \cite[p. 145-146]{Sofonea2012} for a detailed derivation (replacing $\sigma_\tau$ with $\mu\nabla u \cdot \nu - \lambda\beta^2 u$). 
\end{remark}
\begin{remark}\label{remark:comparion}
Both in the usual frictional elastic and viscoelastic problems and in the simplified system considered here, they introduce truncation operators in the friction law and evolution equation for the bonding field (see, e.g., \cite{Sofonea2006}). By introducing the truncation operator $R_1 : \R \rightarrow \R$  
defined as
\begin{align}
R_1(s) = &\begin{dcases}
\  \ L,\  \ \ \text{ if }  s \geq  L,\\
\ \  s,  \ \ \ \:  \text{ if } |s| \leq L,\\
-L,  \ \ \ \text{ if }  s \leq -L,
\end{dcases} 
&\text{with } L>0,
\intertext{ we may reduce Problem \ref{prob:application} to the one studied in \cite{Sofonea2002}. Indeed, with $R_L$ in mind, we may write \hyperref[evol3]{$(E3)$} and \eqref{eq:prob1_friction_eq2}-\eqref{eq:prob1_friction_eq}, respectively, as}
\label{eq:truncated_versions}
&\begin{dcases}
\frac{\partial \beta}{\partial t}  = -\lambda \frac{\beta_+}{1+\beta_+} (R_{L_1}(u))^2,\\
\mu\nabla u \cdot \nu  = -q(\beta) R_{L_2}(u),
\end{dcases}
& \text{for some }L_1,L_2 >0,
\end{align} 
where $q$ is a non-negative stiffness function. Here, $q(\beta) R_{L_2}(u)$ satisfies a Lipschitz bound with respect to both arguments.	Now, Problem \ref{prob:application} with \eqref{eq:truncated_versions} is the system of equations studied in \cite{Sofonea2002}. Moreover, truncated versions of \hyperref[evol1]{$(E1)$} (for $\mathcal{E}_D = 0$), \hyperref[evol2]{$(E2)$}, \hyperref[evol3]{$(E3)$}, and  \eqref{eq:frictionlawused1} have been studied extensively in the usual frictional elastic and viscoelastic problems, see, for instance, \cite{Chau2003,Sofonea2006,Lerguet2008,SofoneaMatei2006},\cite[Section 11.4]{Shillor2004}. Additionally, in \cite[Section 11.4]{Shillor2004}, the part with $\beta^2$ in the friction law is neglected. In \cite{Chau2003,Sofonea2006,SofoneaMatei2006}, they neglect  $\beta^2$ in the friction law but include $\beta^2$ in the contact condition for the normal component of the stress. We have only mentioned quasi-static and static contact problems so far, but the truncations are also used in the dynamic problems, see, e.g., \cite{Chau2003_dynamic,Chau2004,Fernadndez2003_reversibleprocess}.  
\end{remark}
The existence, uniqueness, and regularity result for Problem \ref{prob:application} is a consequence of Theorem \ref{thm:mainresult}. However, for $H$ given by the right-hand side of \hyperref[evol3]{$(E3)$}, we need to restrict the initial bonding field. We therefore postpone this to Corollary \ref{cor:beta012}. We summarize the result for Problem \ref{prob:application} with $H$ given by the right-hand side of \hyperref[evol1]{$(E1)$} or \hyperref[evol2]{$(E2)$} below.
\begin{cor}\label{cor:application}
Assume that \eqref{eq:assumption_g}-\eqref{eq:smallness_app}, \eqref{eq:assumptionmu}, and \eqref{eq:assumptionondata} hold. Let either 
\begin{subequations}
\begin{align}\label{eq:H_E1}
H &= \mathcal{E}_D -   \lambda u^2  \beta,  
\intertext{ or }
H&=-(\lambda u^2 \beta -\mathcal{E}_D )_+.\label{eq:H_E2}
\end{align}
\end{subequations}
Then, there exists a $T>0$ so that $(u,\beta) \in L^\infty_T (H^1_0(\Omega) \cap L^\infty(\Omega)) \times C([0,T];L^2(\Gamma_C)) \cap L^\infty_T L^\infty(\Gamma_C) $ is a unique mild solution to Problem \ref{prob:application}, or more precisely a unique solution to \eqref{eq:variational_form}-\eqref{eq:integralform_beta}.
\end{cor}
\begin{remark}
We observe that \eqref{eq:H_E1}-\eqref{eq:H_E2} is the right-hand side of  \hyperref[evol1]{$(E1)$}-\hyperref[evol2]{$(E2)$}, respectively.
\end{remark}
\begin{remark}
The function spaces in the corollary are defined in Section \ref{sec:funcspace}.
\end{remark}
\begin{proof}
To utilize Theorem \ref{thm:mainresult}, we first have to write \eqref{eq:variational_form} on the same form as \eqref{eq:eq1a}. The next step is to check that the hypotheses of the theorem are fulfilled. We define the functions $\varphi : \Gamma_C \times \R \times \R \times \R \rightarrow \R$ and $\phi: \Gamma_C \times \R \times \R \times \R \rightarrow \R$, respectively, by
\begin{align}\label{eq:varphi_phi}
\varphi(x,\beta, u,  v) &=  g(|u|,\beta)|v| 
&\text{ and }&
&\phi(x,\beta, u,  v) &= -\lambda \beta^2 u v
\end{align}
for all $\beta,u,v \in \R$. Then \eqref{eq:variational_form} has the form of \eqref{eq:eq1a}. It is direct to verify that $\varphi$, $\phi$, and $H$ satisfy \hyperref[assumptionvarphi]{$(H1)$}-\hyperref[assumptionH]{$(H3)$}, respectively, with
\begin{align*}
c_{0\varphi}&= c_{0g}, &c_{1\varphi}= c_{1g} & &c_{2\varphi}= c_{2g}, & & c_{0\beta} = \norm{\mathcal{E}_D}_{L^\infty(\Gamma_C)},
\end{align*}
\begin{align*}
    c_{1\phi}&= c_{2\phi}= c_{3\phi} = c_{1\beta} = c_{2\beta} = c_{3\beta} = \norm{\lambda}_{L^\infty(\Gamma_C)}.
\end{align*}
Thus, it follows by \eqref{eq:smallness_app} that \eqref{eq:smallness} holds. Consequently, we may conclude by Theorem \ref{thm:mainresult} that Problem \ref{prob:application} has a unique mild solution with the desired regularity. 
\end{proof}
As mentioned earlier in this section, the physical constraint on $\beta$ is keeping its range in $[0,1]$ (see, e.g., \cite{Shillor2004} for more details). It is therefore natural to assume that $\beta_0$ has this constraint. To obtain the physical constraint for $\beta$, we need to make a bit stricter assumption on $\beta_0$ (see Assumption \ref{assump:assump}). Additionally, we will see that this implies that \eqref{eq:integralform_beta} with either \eqref{eq:H_E1} or \eqref{eq:H_E2} is irreversible. This coincides with the properties for the existing evolution equations for adhesion (see Remark \ref{remark:thizone}-\ref{remark:props}). Consequently, the extra assumption on $\beta_0$ is justified and of the form:
\begin{assump}\label{assump:assump}
There exists a $\Tilde{c}>0$ so that $\Tilde{c} \leq \beta_0(x) \leq 1$ for all $x\in \Gamma_C$.
\end{assump}
\begin{remark}
If Assumption \ref{assump:assump} holds, the smallness-condition \eqref{eq:smallness_app} is reduced to 
\begin{align}\label{eq:smallness_beta01}
\mu_\ast > c_{2g}c_0 + \norm{\lambda}_{L^\infty(\Gamma_C)}c_0.
\end{align}
\end{remark}
The verification that the solution $\beta$ of the integral equation \eqref{eq:integralform_beta} is decreasing and has the physical constraint $[0,1]$ is summarized in the next corollary.
\begin{cor}\label{lem:beta01}
Assume that \eqref{eq:assumption_g}-\eqref{eq:assumption_EDLG}, \eqref{eq:smallness_beta01}, \eqref{eq:assumptionmu}, \eqref{eq:assumptionondata},  and  Assumption \ref{assump:assump} hold. Then
\begin{itemize}
\item[$(a)$] there exists a $T>0$ such that $(u,\beta) \in  L^\infty_T (H^1_0(\Omega) \cap L^\infty(\Omega)) \times C([0,T];L^2(\Gamma_C)) \cap L^\infty_T L^\infty(\Gamma_C) $ is a unique solution to \eqref{eq:variational_form}-\eqref{eq:integralform_beta} with $H$ given by either \eqref{eq:H_E2} or
\begin{align}\label{eq:H_E1_ED0}
H &=  -   \lambda u^2  \beta.
\end{align}
\end{itemize}
\begin{itemize}
\item[$(b)$] $t\mapsto \beta(\cdot,t)$ is decreasing and $0\leq\beta(x,t) \leq 1$ for all $x\in \Gamma_C$ and a.e. $t\in (0,T)$. 
\end{itemize} 
\end{cor}
\begin{remark}\label{remark:thizone}
If $\beta$ satisfies the properties in Corollary \ref{lem:beta01}, then the integral equations of \hyperref[evol1]{$(E1)$} with $\mathcal{E}_D =0$ and \hyperref[evol2]{$(E2)$} model irreversible processes and $\beta$ satisfies the desired physical constraint $[0,1]$.
\end{remark}
\begin{remark}\label{remark:props}
The properties in Corollary \ref{lem:beta01} are also true in the truncated version of the evolution equations (see, e.g., \cite{Shillor2004,Lerguet2008,SofoneaMatei2006}). In other words, the truncated version of \eqref{eq:integralform_beta} with  \eqref{eq:H_E2} or \eqref{eq:H_E1_ED0} (see Remark \ref{remark:comparion}) also possess these properties.
\end{remark}
\begin{remark}
Taking $\mathcal{E}_D =0$ was used in, e.g., \cite[Section 11.4]{Shillor2004} and \cite{SofoneaMatei2006}. 
\end{remark}
\begin{proof}
The result in $(a)$ follows by  Corollary \ref{cor:application}. 
\\
\indent 
To verify the properties in $(b)$, we consider first \eqref{eq:integralform_beta} with \eqref{eq:H_E1_ED0}.  Applying H\"{o}lder's inequality and Proposition \ref{prop:traceinfty} to \eqref{eq:integralform_beta} with \eqref{eq:H_E1_ED0}, we have for all $x\in \Gamma_C$ that
\begin{align*}
\beta(x,t) &= \beta_0(x) - \int_0^t \lambda \beta(x,s) u(x,s)^2\d s  \geq \beta_0(x) -  T\norm{\lambda}_{L^\infty(\Gamma_C)} \norm{\beta}_{L^\infty_TL^\infty(\Gamma_C)} \norm{u}_{L^\infty_TL^\infty(\Omega)}^2
\end{align*}
for a.e. $t\in (0,T)$.	By the regularity of $(u,\beta)$ and Assumption \ref{assump:assump}, we may choose a $T>0$ small enough 
\begin{align}\label{eq:T_and_beta_0}
T &= \frac{\tilde{c}}{2c^2}&  \text{ so that }&
&\beta(x,t) &\geq \tilde{c}- \frac{\tilde{c}}{2} = \frac{\tilde{c}}{2}  >0
\end{align}
for all $x\in \Gamma_C$ and a.e. $t\in(0,T)$. We have from the equation and Assumption \ref{assump:assump} that $\beta(x,t)\leq \beta_0(x)\leq  1$ for all $x\in \Gamma_C$ and a.e. $t\in (0,T)$ and that $t\mapsto\beta(\cdot,t)$ is decreasing. 
\\
\indent Next, we have that $\mathcal{E}_D \in L^\infty(\Gamma_C)$. Applying H\"{o}lder's inequality and Proposition \ref{prop:traceinfty} to \eqref{eq:integralform_beta} with \eqref{eq:H_E2} yields
\begin{align*}
\beta(x,t) &= \beta_0(x) - \int_0^t(\lambda \beta(x,s) u(x,s)^2-\mathcal{E}_D)_+\d s \\
&\geq \beta_0(x) - T(\norm{\lambda}_{L^\infty(\Gamma_C)} \norm{\beta}_{L^\infty_TL^\infty(\Gamma_C)} \norm{u}_{L^\infty_TL^\infty(\Omega)}^2 + \norm{\mathcal{E}_D}_{L^\infty(\Gamma_C)})
\end{align*}
for all $x\in \Gamma_C$ and a.e. $t\in (0,T)$. Choosing $T>0$ similarly to the above, we may conclude in the same way. \\
\end{proof}
Finally, we state the existence, uniqueness, and regularity result for \eqref{eq:variational_form}-\eqref{eq:integralform_beta} with $H$ given by the right-hand side of \hyperref[evol3]{$(E3)$}. This requires small modifications of the proof of Theorem \ref{thm:mainresult}. We state the changes in the proof of Corollary \ref{cor:beta012}. In the same result, we will also verify that the physical constraint of $\beta$ holds (see Remark \ref{remark:remark123123}-\ref{remark:props2}). In the proof of Corollary \ref{cor:beta012}, we need to restrict the initial bonding field slightly more. 
\begin{assump}\label{assump:assump2}
There exist $\Tilde{c}_1,\Tilde{c}_2>0$ satisfying $0<\Tilde{c}_1 \leq \Tilde{c}_2<1$ so that $\Tilde{c}_1 \leq \beta_0(x) \leq \Tilde{c}_2$ for all $x\in \Gamma_C$.
\end{assump}
\begin{cor}\label{cor:beta012}
    Assume that \eqref{eq:assumption_g}-\eqref{eq:assumption_EDLG}, \eqref{eq:smallness_beta01}, \eqref{eq:assumptionmu}, \eqref{eq:assumptionondata}, and Assumption \ref{assump:assump2} hold. Then
\begin{itemize}
\item[$(a)$] there exists a time $T>0$ such that $(u,\beta) \in  L^\infty_T (H^1_0(\Omega) \cap L^\infty(\Omega)) \times C([0,T];L^2(\Gamma_C)) \cap L^\infty_T L^\infty(\Gamma_C) $ is a unique solution to \eqref{eq:variational_form}-\eqref{eq:integralform_beta} with $H$ given by 
\begin{align}\label{eq:H_E3_3}
        H  &= -\lambda \frac{\beta}{1+\beta} u^2. 
    \end{align}
\end{itemize}
\begin{itemize}
\item[$(b)$] $t\mapsto \beta(\cdot,t)$ is decreasing and $0\leq\beta(x,t) \leq 1$ for all $x\in \Gamma_C$ and a.e. $t\in (0,T)$. 
\end{itemize} 
\end{cor}
\begin{remark}
    We observe that \eqref{eq:H_E3_3} is the right-hand side of \hyperref[evol3]{$(E3)$}.
\end{remark}
\begin{remark}\label{remark:remark123123}
If $\beta$ satisfies the properties in Corollary \ref{lem:beta01}, then the integral equation of \hyperref[evol3]{$(E3)$} models an irreversible process and $\beta$ satisfies the desired physical constraint $[0,1]$.
\end{remark}
\begin{remark}\label{remark:props2}
The properties in Corollary \ref{cor:beta012} are also true in the truncated version of the evolution equations (see, e.g., \cite{Sofonea2006,Sofonea2002}). In other words, the truncated version of \eqref{eq:integralform_beta} with \eqref{eq:H_E3_3} also possess these properties (see Remark \ref{remark:comparion}).
\end{remark}
\begin{proof}
    To prove $(a)$, we need to slightly modify the proof of Lemma \ref{lemma:lambda_fixedpoint} in Step 5.1 of the proof of Theorem \ref{thm:mainresult}. The rest follows directly from Theorem \ref{thm:mainresult}. We therefore only state these changes. We define $\varphi$ and $\phi$ as in \eqref{eq:varphi_phi}, which implies that the hypotheses on  $\varphi$ and $\phi$ hold with 
    \begin{align*}
        c_{0\varphi}&= c_{0g}, &c_{1\varphi}= c_{1g} & &c_{2\varphi}= c_{2g}, &&     c_{1\phi}= c_{2\phi}= c_{3\phi} = \norm{\lambda}_{L^\infty(\Gamma_C)}.
    \end{align*}
   \indent We note that in this part of the proof of Theorem \ref{thm:mainresult}, we use induction to obtain uniform estimates. In a similar fashion as in the proof of Theorem \ref{thm:mainresult}, we only state the initial case, as the induction step follows the same approach. For simplicity in notation, we denote $\beta = \beta^1$ and $u = u^1$, where $u^1 \in  L^\infty_T (H^1_0(\Omega) \cap L^\infty(\Omega))$ is a solution to \eqref{eq:u_n} for $n=1$.  We first define the complete metric space $X_T$ by
    \begin{align}\label{eq:metric_sapce}
    	X_T := \{ v \in C([0,T];L^2(\Gamma_C)) :  0 \leq v(x,t) \leq 1  \text{ for all } x\in \Gamma_C \text{ and a.e. } t\in (0,T) \}
    \end{align}
    and the application $\Lambda :X_T  \rightarrow X_T$ 
    \begin{align}\label{eq:Lambda3}
        \Lambda \beta(t) = \beta_0 - \lambda \int_0^t\frac{\beta(s)}{1+\beta(s)} (u(s))^2 \d s.
    \end{align}
    In comparison to Step i in the proof of Lemma \ref{lemma:lambda_fixedpoint}, we only need to verify that $0 \leq \Lambda \beta(x,t) \leq 1$ for all $x\in \Gamma_C$ and a.e. $t\in (0,T)$. For $\beta \in X_T$, Proposition \ref{prop:traceinfty} in  \eqref{eq:lamda_line1} yields
    \begin{align*}
        \Lambda \beta(x,t)
        &\leq \beta_0(x)  + T c^2 \leq \Tilde{c}_2 + T c^2,
    \end{align*}
    where we may choose a $T>0$ so that 
    \begin{align*}
         \Lambda \beta(x,t) \leq 1
    \end{align*}
    for all $x\in \Gamma_C$ and a.e. $t\in (0,T)$. Verifying that $\Lambda \beta(x,t) \geq 0$ for all $x\in \Gamma_C$ and a.e. $t\in (0,T)$ follows by Proposition \ref{prop:traceinfty}, Assumption \ref{assump:assump2}, and $\beta \in X_T$. Indeed,
    \begin{align*}
         \Lambda \beta(x,t) &=  \beta_0(x) - \lambda \int_0^t\frac{\beta(x,s)}{1+\beta(x,s)} (u(x,s))^2 \d s \\ &\geq \beta_0(x) - T\norm{\lambda}_{L^\infty(\Gamma_C)} \norm{u}_{L^\infty_TL^\infty(\Omega)}^2
    \end{align*}
    for all $x\in \Gamma_C$ and a.e. $t\in (0,T)$. We may therefore choose a $T>0$ such that $\Lambda \beta(x,t)  \geq 0$ for all $x\in \Gamma_C$ and a.e. $t\in (0,T)$ (similar to \eqref{eq:T} and \eqref{eq:T_and_beta_0}). The continuity follows directly by the fact that $\frac{\beta}{1+\beta} \leq 1$ and the last part of Step i in the proof of Lemma \ref{lemma:lambda_fixedpoint}.
    \indent
    Next,  for any $\beta_1, \beta_2 \in X_T$, we have 
    \begin{align*}
        \big|\frac{\beta_1}{1+\beta_1} - \frac{\beta_2}{1+\beta_2} \big| &= \big|\frac{\beta_1(1+\beta_2) - \beta_2(1+\beta_1)}{(1+\beta_1)(1+\beta_2)} \big| \leq |\beta_1+\beta_1\beta_2 - \beta_2 - \beta_2\beta_1 | = |\beta_1 - \beta_2|.
    \end{align*}
    We may therefore conclude as in the proof in Lemma \ref{lemma:lambda_fixedpoint} that $\beta \in X_T$ is a unique fixed point to \eqref{eq:Lambda3}. 
     \\ 
    \indent
    To utilize the rest of the proof of Theorem \ref{thm:mainresult}, we verify that the hypothesis \hyperref[assumptionH]{$(H3)$} for $H$ given by \eqref{eq:H_E3_3} holds. We can easily find that they are satisfied with $c_{0\beta} =0$ and $c_{1\beta}=c_{2\beta}= 2c_{3\beta} = \norm{\lambda}_{L^\infty(\Gamma_C)}$.
    \\
    \indent For the remaining part (part $(b)$), we have from part $(a)$ that $0\leq \beta(x,t) \leq 1$ for all $x\in \Gamma_C$ and a.e. $t\in (0,T)$. Therefore,  we only have left to show that $t\mapsto \beta(\cdot,t)$ is decreasing for a.e. $t\in (0,T)$. But this follows directly from the equation. 
\end{proof}

\bibliographystyle{acm}
\bibliography{antiplane}

\end{document}